\documentclass[12pt, reqno, oneside, a4paper]{amsart}
\usepackage[a4paper]{geometry}
\usepackage[utf8]{inputenc}
\usepackage{amssymb, amsmath, amsthm, esint, enumitem,dsfont,tikz}
\usetikzlibrary{angles,quotes}
\usepackage{latexsym}
\usepackage{graphicx}
\usepackage{microtype}

\usepackage{hyperref}
\hypersetup{allcolors=blue,colorlinks=true}
\allowdisplaybreaks 
\usepackage{mathtools}
\mathtoolsset{showonlyrefs}

\theoremstyle{plain}
\newtheorem{defi}{Definition}[section]
\newtheorem{thm}[defi]{Theorem}
\newtheorem{ex}[defi]{Example}
\newtheorem{rem}[defi]{Remark}
\newtheorem{prop}[defi]{Proposition}
\newtheorem{lemma}[defi]{Lemma}
\newtheorem{cor}[defi]{Corollary}
\newtheorem{question}[defi]{Question}

\newcommand{\R}{\mathbb{R}}
\newcommand{\N}{\mathbb{N}}
\newcommand{\Z}{\mathbb{Z}}
\renewcommand{\O}{\mathcal{O}}

\renewcommand{\L}{\mathcal{L}}
\newcommand{\diam}{\mathrm{diam}}

\renewcommand{\epsilon}{\varepsilon}
\renewcommand{\emptyset}{\varnothing}

\title[On the minimisation of eigenvalues of the Laplacian]{On the Isoperimetric and Isodiametric inequalities and the minimisation of eigenvalues of the Laplacian}
\author{Sam Farrington}
\address{Sam Farrington\\ Department of Mathematical Sciences \\ Durham University \\
Mathematical Sciences and Computer Science Building \\
Upper Mountjoy Campus \\ Stockton Road \\
Durham DH1 3LE \\
United Kingdom.
}
\email{sam.farrington@durham.ac.uk}  

\date{\today}

\begin{document}

\begin{abstract}
    We consider the problem of minimising the $k$-th eigenvalue of the Laplacian with some prescribed boundary condition over collections of convex domains of prescribed perimeter or diameter. It is known that these minimisation problems are well-posed for Dirichlet eigenvalues in any dimension $d\geq 2$ and any sequence of minimisers converges to the ball of unit perimeter or diameter respectively as $k\to +\infty$. In this paper, we show that the same is true in the case of Neumann eigenvalues under diameter constraint in any dimension and under perimeter constraint in dimension $d=2$. We also consider these problems for mixed Dirichlet-Neumann eigenvalues, under an additional geometric constraint, and discuss some applications of our proof techniques.

\end{abstract}

\maketitle

\section{Introduction}

Given $\Omega \subset \mathbb{R}^{d}$ a bounded convex domain, it is well-known that the Dirichlet $-\Delta_{\Omega}^{D}$ and Neumann $-\Delta_{\Omega}^{N}$ Laplacians acting on $\mathcal{L}^{2}(\Omega)$ have discrete spectra, each consisting of a sequence of eigenvalues accumulating only at $+\infty$. We denote the Dirichlet eigenvalues by \begin{equation}
    0< \lambda_{1}(\Omega) < \lambda_{2}(\Omega) \leq \lambda_{3}(\Omega) \leq \cdots \uparrow + \infty
\end{equation} and the Neumann eigenvalues by \begin{equation}
    0 = \mu_{1}(\Omega) < \mu_{2}(\Omega) \leq \mu_{3}(\Omega) \leq \cdots \uparrow + \infty.
\end{equation} Moreover, it is well-known that these eigenvalues obey Weyl's law, which asserts that \begin{equation}\label{eq:Weyl_law_orig}
    \lambda_{k}(\Omega) \sim \mu_{k}(\Omega) \sim 4\pi^{2} \left(\frac{k}{\omega_{d}|\Omega|}\right)^{2/d} =: \frac{W_{d}}{|\Omega|^{2/d}}k^{2/d}, \enspace \text{as} \enspace k\uparrow +\infty,
\end{equation} where $|\Omega|$ is the $d$-dimensional volume of $\Omega$ and $\omega_{d}$ is the volume of the $d$-dimensional unit ball. From Weyl's law, if one knows either the entire Dirichlet spectrum or the entire Neumann spectrum of $\Omega$, then one can determine the volume of $\Omega$.

Naïvely, Weyl's law and the isoperimetric/isodiametric inequality together suggest that if one minimises either Dirichlet or Neumann eigenvalues over the collection of bounded convex domains of a given perimeter/diameter then for large $k$ minimisers should be close to the ball, i.e. the domain with the largest volume. To be clear, by perimeter here we mean the $(d-1)$-dimensional Hausdorff measure of the boundary $\partial \Omega$, which we denote by $|\partial \Omega|$. 

In this vein, one can consider the four following spectral shape optimisation problems: \begin{gather}
\inf\left\lbrace \lambda_{k}(\Omega) : \Omega \subset \mathbb{R}^{d} \; \text{bounded, convex}, \; |\partial \Omega| = 1 \right\rbrace, \label{prob:dir_perim} \\ 
    \inf\left\lbrace \lambda_{k}(\Omega) : \Omega \subset \mathbb{R}^{d} \; \text{bounded, convex}, \; \diam(\Omega) = 1 \right\rbrace, \label{prob:dir_diam} \\
    \inf\left\lbrace \mu_{k}(\Omega) : \Omega \subset \mathbb{R}^{d} \; \text{bounded, convex}, \; |\partial \Omega| = 1 \right\rbrace, \label{prob:neu_perim} \\ 
    \inf\left\lbrace \mu_{k}(\Omega) : \Omega \subset \mathbb{R}^{d} \; \text{bounded, convex}, \; \diam(\Omega) = 1 \right\rbrace. \label{prob:neu_diam} 
\end{gather} 

For each of the above problems we will discuss when one has existence of minimisers and, if so, the geometric behaviour of minimisers as $k\to +\infty$. To do this we need to introduce a notion of convergence onto the collection of bounded convex domains. We use the Hausdorff metric here, which is defined as \begin{equation}\label{eq:Hausdorff_metric}
    d_{H}(\Omega_{1},\Omega_{2}) := \max \left\lbrace \sup_{x\in\Omega_{1}} \inf_{y\in \Omega_{2}} \lVert x-y\rVert_{2}, \sup_{x\in \Omega_{2}}\inf_{y\in \Omega_{1}} \lVert x-y\rVert_{2}\right\rbrace
\end{equation} for two bounded convex domains $\Omega_{1},\Omega_{2} \subset \mathbb{R}^{d}$, where $\lVert \cdot \rVert_{2}$ is the standard Euclidean norm. In this paper, the convergence is always up to rigid planar motions where necessary due to the invariance of Dirichlet and Neumann eigenvalues under rigid planar motions.

The shape optimisation problem \eqref{prob:dir_perim} was first considered in the case of $k=1$ and for planar domains by Courant in \cite{Courant-1918}, where it was shown that the ball is the minimiser. This can also be shown using the Faber-Krahn \cite[Thm. 3.2.1.]{Henrot-2006} and isoperimetric inequalities. The existence of a minimiser to \eqref{prob:dir_perim} for any $k\in\mathbb{N}$ when $d = 2$ was given by van den Berg and Iversen in \cite{van-den-Berg-Iversen-2013} under the more general conditions of minimising the $k$-th Dirichlet eigenvalue among non-empty bounded open sets of unit perimeter without any convexity or connectivity constraints. They showed that a minimiser in this case is necessarily convex. Bucur and Freitas in \cite{Bucur-Freitas-2013} later showed, when $d=2$, that any sequence $\Omega_{k}^{*}$ of minimisers to \eqref{prob:dir_perim} Hausdorff converges to the ball of unit perimeter as $k \to +\infty$. For $d\geq 3$, van den Berg in \cite{van-den-Berg-2015} deduced existence of minimisers and the same asymptotic behaviour as in the two-dimensional case.

\begin{thm}[{\cite[Thm. 1, adapted]{van-den-Berg-2015}}]\label{thm:Dirichlet_perim}
For each $k\geq 1$, there exists a minimiser to \eqref{prob:dir_perim}. Moreover, any sequence of minimisers Hausdorff converges to the ball of unit perimeter as $k\to +\infty$. 
 \end{thm}

In the same paper, \cite{van-den-Berg-2015}, van den Berg also considered the shape optimisation problem \eqref{prob:dir_diam} and deduced an analogous result to Theorem \ref{thm:Dirichlet_perim} in this case.

\begin{thm}[{\cite[Thm. 1, adapted]{van-den-Berg-2015}}]\label{thm:Dirichlet_diam}
For each $k\geq 1$, there exists a minimiser to \eqref{prob:dir_diam}. Moreover, any sequence of minimisers Hausdorff converges to the ball of unit diameter as $k\to +\infty$. 
 \end{thm}

Bogosel, Henrot and Lucardesi also studied the shape optimisation problem \eqref{prob:dir_diam} in \cite{Bogosel-Henrot-Lucardesi-2018} and deduced that the ball is only a minimiser for finitely many $k \in \mathbb{N}$ and minimisers are necessarily bodies of constant width.

In light of Theorems \ref{thm:Dirichlet_perim} and \ref{thm:Dirichlet_diam}, the aforementioned naïve notion of minimisers for large $k$ being close to a ball holds for Dirichlet eigenvalues. However, for the case of Neumann eigenvalues, the minimisation problem \eqref{prob:neu_perim} is well-known to be ill-posed for any $k\geq 2$ and $d\geq 3$ as the infimum is zero in this case but $\mu_{2}(\Omega) > 0$ for any bounded convex domain $\Omega$. Hence, the naïve philosophy that motivated these questions at the beginning of this paper fails here. The fact that this infimum is zero can easily be deduced by considering the sequence of cuboids \begin{equation}
    (0,\epsilon) \times  \cdots \times (0,\epsilon) \times \left(0,\tfrac{1}{2(d-1)}(\epsilon^{2-d}-2\epsilon)\right) \subset \mathbb{R}^{d}
\end{equation} as $\epsilon \downarrow 0$. When $d=2$, it was shown by van den Berg, Bucur and Gittins in \cite[Thm 3.2.]{van-den-Berg-Bucur-Gittins-2016} that the minimisation problem \begin{equation}
    \inf \left\lbrace \mu_{k}(\Omega) : \Omega \subset \mathbb{R}^{2} \enspace \text{rectangle}, \, |\partial \Omega| = 1 \right\rbrace
\end{equation} has a minimiser for all $k\geq 3$ and any sequence of minimisers Hausdorff converges to the square of unit perimeter as $k\to +\infty$. The analogous result is true in the more general setting of planar convex domains as we now state.

\begin{thm}\label{thm:neumann_perim}
When $d=2$, for each $k\geq 3$ there exists a minimiser to \eqref{prob:neu_perim}. Moreover, any sequence of minimisers Hausdorff converges to the ball of unit perimeter as $k\to +\infty$. 
\end{thm}

Our methods in this paper are asymptotic and in fact only assert that minimisers exist for $k$ sufficiently large. One can establish that they exist for each $k\geq 3$, as in the statement of Theorem \ref{thm:neumann_perim}, and not for $k=2$ from a simple adaptation of the methods in \cite[\S 2]{Cavallani2023}, see Lemma \ref{lem:2d_existence_lemma}. In \cite[\S 2]{Cavallani2023}, the authors study the `interior problem' \begin{equation}
    \inf\lbrace \mu_{k}(\Omega) : \Omega \subset D, \, \Omega \; \text{convex domain} \rbrace,
\end{equation} where $D$ is a fixed bounded convex domain in $\mathbb{R}^{2}$, and give sufficient and necessary conditions for the existence of minimisers. On the other hand, our methods allow one to show that any sequences of minimisers to the `interior problem' must necessarily Hausdorff converge to $D$ as $k\to +\infty$.

In contrast to the perimeter case, we have that for any $k \geq 2$ and $d\geq 2$ the infimum in \eqref{prob:neu_diam} is non-zero from the Payne-Weinberger inequality, see \cite{Payne-Weinberger-1960} and \cite{Bebendorf-2003}, which asserts for any bounded convex domain $\Omega \subset \mathbb{R}^{d}$ \begin{equation}\label{eq:payne-weinberger}
    \mu_{2}(\Omega) > \frac{\pi^{2}}{\diam(\Omega)^{2}}.
\end{equation} This inequality is sharp and attained in the limit by the sequence of cuboids \begin{equation}
    (0,\epsilon) \times \cdots \times (0,\epsilon) \times \left(0,\sqrt{1-(d-1)\epsilon^{2}}\right) \subset \mathbb{R}^{d}
\end{equation} as $\epsilon \downarrow 0$, in the case where $\diam(\Omega) = 1$. It is not immediately clear if and when the shape optimisation problem \eqref{prob:neu_diam} admits minimisers. However, minimisers do eventually exist for $k$ sufficiently large and one can obtain the asymptotic behaviour of minimisers as $k\to +\infty$. In particular, in dimension two, as in the perimeter case, we can again say that they exist for all $k\geq 3$ adapting the methods in \cite[\S 2]{Cavallani2023}.

\begin{thm}\label{thm:neumann_diam}
    For any $d\geq 2$, there exists a constant $N_{d} \in \mathbb{N}$ such that for all $k\geq N_{d}$ there exists a minimiser to \eqref{prob:neu_diam}. Moreover, any sequence $\Omega_{k}^{*}$ of minimisers Hausdorff converges to the ball of unit diameter as $k\to +\infty$. When $d=2$, one has $N_{2} = 3$ is the lowest possible value of $N_{2}$.
\end{thm}

The proofs of Theorems \ref{thm:neumann_perim} and \ref{thm:neumann_diam} rest on proving a suitably good family of upper bounds for the Neumann eigenvalue counting function \begin{equation}\label{eq:Neumann_counting}
    \mathcal{N}_{\Omega}^{N}(\alpha) := \# \lbrace k \in \mathbb{N} : \mu_{k}(\Omega) < \alpha \rbrace 
\end{equation} for a bounded convex domain $\Omega\subset \mathbb{R}^{d}$. This family of bounds provides the existence of minimisers and gives their asymptotic behaviour. We also give a family of lower bounds on the Dirichlet eigenvalue counting function \begin{equation}\label{eq:Dirichlet_counting}
    \mathcal{N}_{\Omega}^{D}(\alpha) := \# \lbrace k \in \mathbb{N} : \lambda_{k}(\Omega) < \alpha \rbrace ,
\end{equation} from which one can prove a `generalisation' of Weyl's law for bounded convex domains which we will discuss in Section \ref{sec:other_applications}.

The family of upper bounds on the Neumann counting functions can also be used to give a variation of \eqref{prob:neu_perim} for which one does have the existence of minimisers and for which any sequence of minimisers Hausdorff converges to the ball of unit perimeter as $k\to +\infty$. The philosophy is if we don't allow the domains in the collection under consideration to grow too quickly with $k$ in terms of their diameter then we obtain non-degenerate asymptotic behaviour.

\begin{thm}\label{thm:neumann_per_diam}
    For any $d\geq 3$ and any $f: \mathbb{N} \to \mathbb{R}_{>0}$ with $1\ll f(k) \ll k^{1/d(d-1)}$, there exists a constant $N_{d,f} \in \mathbb{N}$ such that for all $k\geq N_{d,f}$ there exists a minimiser to \begin{equation}\label{eq:neu_per_diam_prob}
    \inf \left\lbrace \mu_{k}(\Omega) : \Omega \subset \mathbb{R}^{d} \enspace \text{bounded convex domain}, \, |\partial\Omega| \leq 1,\, \diam(\Omega) \leq f(k)\right\rbrace.
    \end{equation} Moreover, any sequence $\Omega_{k}^{*}$ of minimisers Hausdorff converges to the ball of unit perimeter as $k\to +\infty$.
\end{thm}

From the foregoing discussion, the perimeter constraint minimisation problems \eqref{prob:dir_perim} for Dirichlet eigenvalues and \eqref{prob:neu_perim} for Neumann eigenvalues are well-posed and ill-posed respectively for any $k\geq 2$ and any dimension $d\geq 3$. In Section \ref{sec:zaremba}, we consider a variation of the perimeter constraint eigenvalue minimisation problem in the case of mixed Dirichlet-Neumann, so-called Zaremba eigenvalues in which the shape optimisation problem exhibits the same behaviour as that of \eqref{prob:dir_perim}, see Theorem \ref{thm:Dirichlet_perim}. For this we introduce an additional geometric constraint on the collection of convex domains which yields a canonical way of prescribing the mixed boundary conditions and allows us to obtain eigenvalue bounds.

\begin{rem}
     The results of Theorems \ref{thm:neumann_perim}, \ref{thm:neumann_diam} and \ref{thm:neumann_per_diam} extend to the case of Robin eigenvalues, with positive Robin parameter as the Robin eigenvalue counting function is bounded above by the Neumann eigenvalue counting function for any bounded domain $\Omega$, see for example \cite[Thm. 3.2.9.]{Levitin-Polterovich-Mangoubi-2023}. Moreover, one can show the existence of minimisers for all $k\geq 1$ using Proposition 2.3 in \cite{Antunes-Freitas-Kennedy-2013} and the lower semi-continuity of Robin eigenvalues under Hausdorff convergence of bounded convex domains, see \cite[Prop. 3.1.]{Cito-2019}.
\end{rem}

\textbf{Plan of the paper:} In Section \ref{sec:counting_bounds}, we prove upper bounds on the Neumann eigenvalue counting function and lower bounds on the Dirichlet eigenvalue counting function for bounded convex domains. Using these bounds, we prove Theorems \ref{thm:neumann_perim}, \ref{thm:neumann_diam} and \ref{thm:neumann_per_diam} in Section \ref{sec:main_proofs}. In Section \ref{sec:other_applications}, we discuss applications of these bounds with regards to the geometric stability of Weyl's law. In Section \ref{sec:zaremba}, we consider some cases where one mixes Dirichlet and Neumann boundary conditions. In Appendix \ref{sec:appendix}, we present an open question motivated by the proof techniques used in this paper and its relation with Pólya's eigenvalue conjecture.

\section{Bounding Dirichlet and Neumann counting functions}\label{sec:counting_bounds} 

For a bounded convex domain $\Omega \subset \mathbb{R}^{d}$, recall the definition of its Neumann eigenvalue counting function $\mathcal{N}_{\Omega}^{N}$ from \eqref{eq:Neumann_counting} and its Dirichlet eigenvalue counting function $\mathcal{N}_{\Omega}^{D}$ from \eqref{eq:Dirichlet_counting}. We aim to prove a suitably good family of upper bounds on $\mathcal{N}_{\Omega}^{N}$, which in turn give us lower bounds on the Neumann eigenvalues of $\Omega$, and a suitably good family of lower bounds on $\mathcal{N}_{\Omega}^{D}$, which in turn give us upper bounds on the Dirichlet eigenvalues of $\Omega$. 

In order to prove these bounds, we need to divulge briefly into some convex geometry. We remark that all the results in convex geometry stated in this paper are generally stated for compact convex sets but we given the equivalent statement here for bounded convex domains for simplicity of exposition. A well-known result in this area is the Minkowski-Steiner theorem, which gives an expression of the volume of fattening a bounded convex domain by taking the Minkowski sum with a ball of radius $\delta > 0$, in terms of geometric quantities associated with $\Omega$. 

From here forward, let $\mathcal{O}_{d}$ denote the collection of bounded convex domains endowed with the Hausdorff topology induced by the metric given by \eqref{eq:Hausdorff_metric}.

\begin{thm}[Minkowski-Steiner, {see \cite[\S 6]{Gruber-2007}}] Let $\Omega$ be a $d$-dimensional convex domain and $\delta > 0$. Then there exist continuous maps $s_{2},s_{3},\ldots ,s_{d-1}: \mathcal{O}_{d} \to \mathbb{R}$ called the quermassintegrals of $\Omega$ such that \begin{equation}
    |\Omega+\delta \mathbb{B}_{d}| = |\Omega| +  |\partial \Omega|\delta + \sum_{j=2}^{d-1} \binom{d}{j} s_{j}(\Omega)\delta^{j} + \omega_{d}\delta^{d},
\end{equation} where $\mathbb{B}_{d}$ is the $d$-dimensional unit ball. In particular, as a map $\mathcal{O}_{d} \times (0,+\infty) \to \mathbb{R}$, $|\Omega+\delta \mathbb{B}_{d}|$ is continuous.
    
\end{thm}

It is now worth noting some properties of quermassintegrals. Namely, they are monotone with respect to inclusion, i.e. $s_{j}(\Omega_{1})\leq s_{j} (\Omega_{2})$ for all $j$ and all $\Omega_{1}\subset \Omega_{2}$, and are continuous in the Hausdorff topology, see \cite[\S 6]{Gruber-2007}. The monotonicity property will allow us to obtain bounds on the Neumann eigenvalue counting function which are monotone with respect to domain inclusion.

We are now ready to state and prove a family of upper bounds for the Neumann eigenvalue counting function. The proof is originally inspired by the proof of Proposition A.1. in \cite{Gittins-Lena-2020}, whereby the authors give an upper bound on the Neumann counting function of a $C^{2}$ convex domain. One key aspect is the leading term of their bound is not the leading term in Weyl's law which we need for our proofs. In fact, we are unable to provide bounds with the Weyl leading term here, however we can construct upper bounds on the Neumann counting function whose leading term is arbitrarily close to the Weyl term.

\begin{prop}\label{prop:Neumann_counting function_upper_bound}
    For any $n\in \mathbb{N}$, $\Omega \in \O_{d}$ and $\alpha > 0$,  \begin{equation}
        \mathcal{N}_{\Omega}^{N}(\alpha)  \leq \frac{ n|\Omega|}{(\mu_{n+1}^{*})^{d/2}} \alpha^{d/2}  + r_{n}(\Omega;\alpha),
    \end{equation} where \begin{equation}
    \begin{split}
        r_{n}(\Omega;\alpha) & =  \left(\frac{\kappa_{n}}{\sqrt{\mu_{n+1}^{*}}}\right)^{d-1}\left(2\kappa_{n}+3 \right)d^{1/2}|\partial\Omega| \alpha^{(d-1)/2} \\ & \qquad + \sum_{j=2}^{d-1}  \binom{d}{j} (4d)^{j/2} \left(\frac{\kappa_{n}}{\sqrt{\mu_{n+1}^{*}}}\right)^{d-j}s_{j}(\Omega)\alpha^{(d-j)/2} +(4d)^{d/2}\omega_{d},   
    \end{split}
    \end{equation} $\mu_{n+1}^{*}$ denotes the $(n+1)$-th Neumann eigenvalue of the $d$-dimensional unit cube, $s_{j}(\Omega)$ denotes the $j$-th quermassintegral of $\Omega$ from the Minkowski-Steiner formula and $\kappa_{n} = \left\lceil \pi^{-1}\sqrt{d\mu_{n+1}^{*}}\,\right\rceil$. Moreover, the remainder $r_{n}(\Omega;\alpha)$ is monotone with respect to inclusion of convex domains. 
\end{prop}

\begin{proof}
    Fix $\delta > 0$ and $n \in \N$. For $m \in \Z^{d}$, let $Q_{m,\delta} := \delta(m+(0,1)^{d})$. Note that \begin{equation}
        \mathcal{N}_{Q_{m,\delta}}^{N}(\delta^{-2}\mu_{n+1}^{*}) \leq n
    \end{equation} by the definition of the Neumann counting function and the scaling property of Neumann eigenvalues under homothety. Setting $\mathcal{I}_{\delta} := \lbrace m \in \mathbb{Z}^{d} : Q_{m,\delta} \cap \Omega = Q_{m,\delta}\rbrace$, we immediately see that $\# \mathcal{I}_{\delta} \leq \delta^{-d}|\Omega|$ as for any $m\in \mathcal{I}_{\delta}$ we must have $Q_{m,\delta} \subset \Omega$. Then define $\Omega_{\delta}^{i}:= \bigcup_{m\in \mathcal{I}_{\delta}} Q_{m,\delta}$. Taking $\kappa_{n}\in \N$ as given in the statement of the proposition, let \begin{equation}
        \mathcal{J}_{\delta} := \lbrace m \in \mathbb{Z}^{d} : Q_{m,\kappa_{n}^{-1}\delta} \cap \Omega \neq \emptyset, Q_{m,\kappa_{n}^{-1}\delta} \cap \Omega_{\delta}^{i} = \emptyset \rbrace, \quad \Omega_{\kappa_{n}^{-1}\delta}^{o} := \Omega \cap \bigcup_{m\in \mathcal{J}_{\delta}} Q_{m,\delta}.
    \end{equation} As $\kappa_{n}$ is a positive integer and by construction, we see that $\Omega_{\delta}^{i} \cap \Omega_{\delta}^{o} = \emptyset$, as $\kappa_{n}^{-1} \delta \Z \supset \delta \Z$, and that $\Omega_{\delta}^{i} \cup \Omega_{\kappa_{n}^{-1}\delta}^{o} = \Omega$ up to a set of measure zero. We now argue that for any $m\in \mathcal{J}_{\delta}$, $Q_{m,\kappa_{n}^{-1}\delta}$ must be a subset of the region \begin{equation}
        \partial \Omega_{\delta,\kappa_{n}} := \lbrace x \in \Omega : d(x,\partial \Omega) \leq (2+\kappa_{n}^{-1}) \delta d^{1/2}  \rbrace \cup  \lbrace x \in \mathbb{R}^{d}\backslash \Omega : d(x,\partial \Omega) \leq 2(\kappa_{n})^{-1}\delta d^{1/2}  \rbrace.
    \end{equation} Firstly suppose that \begin{equation}
        Q_{m,\kappa_{n}^{-1}\delta } \cap \lbrace x \in \Omega : d(x,\partial \Omega) > (2+\kappa_{n}^{-1}) \delta d^{1/2}  \rbrace \neq \emptyset.
    \end{equation} Then we see that $d(m\kappa_{n}^{-1}\delta,\partial \Omega) > 2\delta d^{1/2}$. Now let $m^{*} \in \Z^{d}$ be the unique integer lattice point such that $Q_{m^{*},\delta}\supset Q_{m,\kappa_{n}^{-1}\delta }$. Then one easily sees that $\lVert m\kappa_{n}^{-1} \delta - m^{*} \delta \rVert_{2} \leq \delta d^{1/2}$ and so one must have that $d(m^{*}\delta, \partial \Omega) > \delta d^{1/2}$. But this implies that $m^{*} \in \mathcal{I}_{\delta}$ which implies that $m\not \in \mathcal{J}_{\delta}$ and we have a contradiction. Now suppose that \begin{equation}
        Q_{m,\kappa_{n}^{-1}\delta } \cap \lbrace x \in \mathbb{R}^{d}\backslash \Omega : d(x,\partial \Omega) > 2(\kappa_{n})^{-1}\delta d^{1/2}  \rbrace \neq \emptyset.
    \end{equation} Then we have that $d(m\kappa_{n}^{-1}\delta,\partial \Omega) > \delta d^{1/2}$, but again this contradicts $m\in \mathcal{J}_{\delta}$ as in this case $Q_{m,\kappa_{n}^{-1}\delta} \cap \Omega = \emptyset$. Hence, we must indeed have that 
    $Q_{m,\kappa_{n}^{-1}\delta} \subset \partial \Omega_{\delta,\kappa_{n}}$ if $m\in \mathcal{J}_{\delta}$.
    
    By the Minkowski-Steiner formula and the convexity of $\Omega$, we see that \begin{equation}
       |\partial \Omega_{\delta,\kappa_{n}}| \leq  \left(2+3\kappa_{n}^{-1}\right)d^{1/2}|\partial\Omega|\delta  + \sum_{j=2}^{d-1} \binom{d}{j} (4d)^{j/2} (\kappa_{n})^{-j} s_{j}(\Omega)\delta^{j} +(4d)^{d/2}\omega_{d}\delta^{d}.
    \end{equation} We then can immediately deduce a bound on the cardinality of $\mathcal{J}_{\delta}$: \begin{equation}\begin{split}
        \# \mathcal{J}_{\delta} & \leq (\kappa_{n})^{d}\delta^{-d}|\partial \Omega_{\delta,\kappa_{n}}| \\
        & \leq  \left(2(\kappa_{n})^{d}+3(\kappa_{n})^{d-1}\right)d^{1/2}|\partial\Omega|\delta^{-d+1} \\ & \qquad +  \sum_{j=2}^{d-1} \binom{d}{j}(4d)^{j/2} (\kappa_{n})^{d-j} s_{j}(\Omega)\delta^{-d+j}+(4d)^{d/2}\omega_{d}.
    \end{split}
    \end{equation} Observing that $Q_{m,\kappa_{n}^{-1}\delta} \cap \Omega$ is convex, by our choice of $\kappa_{n}$ and the Payne-Weinberger inequality, see \eqref{eq:payne-weinberger}, \begin{equation} 
        \mu_{2}(Q_{m,\kappa_{n}^{-1}\delta} \cap \Omega) \geq \delta^{-2}\mu_{n+1}^{*},
    \end{equation} and so $\mathcal{N}_{Q_{m,\kappa_{n}^{-1}\delta} \cap \Omega}^{N}(\delta^{-2}\mu_{n+1}^{*}) = 1$. By the variational characterisation of Neumann eigenvalues, see \cite[Thm. 3.1.11.]{Levitin-Polterovich-Mangoubi-2023}, it is straightforward to verify that $\mu_{k}(\Omega_{\delta}^{i} \cup \Omega_{\kappa_{n}^{-1}\delta}^{o}) \leq \mu_{k}(\Omega)$ for all $k\in \mathbb{N}$, and so it suffices to bound the Neumann eigenvalue counting function of $\Omega_{\delta}^{i} \cup \Omega_{\kappa_{n}^{-1}\delta}^{o}$. Taking $\delta =\alpha^{-1/2}(\mu_{n+1}^{*})^{1/2}$, we see that \begin{equation}
    \begin{split}
        \mathcal{N}_{\Omega_{\delta}^{i} \cup \Omega_{\kappa_{n}^{-1}\delta}^{o}}^{N}(\alpha) & \leq \sum_{m \in \mathcal{I}_{\delta}} \mathcal{N}_{Q_{m,\delta}}^{N}(\alpha) + \sum_{m \in \mathcal{J}_{\delta}} \mathcal{N}_{Q_{m,\kappa^{-1}\delta}\cap \Omega}^{N}(\alpha) \\
        & \leq \frac{n|\Omega|}{(\mu_{n+1}^{*})^{d/2}}\alpha^{d/2} + \left(\frac{\kappa_{n}}{\sqrt{\mu_{n+1}^{*}}}\right)^{d-1}\left(2\kappa_{n}+3\right) d^{1/2} |\partial\Omega| \alpha^{(d-1)/2} \\ 
        & \qquad + \sum_{j=2}^{d-1} \binom{d}{j} (4d)^{j/2} \left(\frac{\kappa_{n}}{\sqrt{\mu_{n+1}^{*}}}\right)^{d-j}s_{j}(\Omega)\alpha^{(d-j)/2} + (4d)^{d/2}\omega_{d}, \\
        & = \frac{n|\Omega|}{(\mu_{n+1}^{*})^{d/2}}\alpha^{d/2} + r_{n}(\Omega;\alpha),
    \end{split}
    \end{equation} which completes the proof.
\end{proof}

\begin{rem}
    To make the upper bound in Proposition \ref{prop:Neumann_counting function_upper_bound} explicit relies on the computability of the quermassintegrals of the domain. In dimension two, the bound simply reads \begin{equation}
        \mathcal{N}_{\Omega}^{N}(\alpha) \leq \frac{n|\Omega|}{\mu_{n+1}^{*}}\alpha + \sqrt{\frac{2}{\mu_{n+1}^{*}}}\left\lceil \frac{\sqrt{2\mu_{n+1}^{*}}}{\pi} \right\rceil\left(2\left\lceil \frac{\sqrt{2\mu_{n+1}^{*}}}{\pi} \right\rceil+3\right)|\partial \Omega|\sqrt{\alpha} + 8\pi.
    \end{equation}
\end{rem}

One can also play the same game with Dirichlet counting functions and prove a lower bound analogously to the proof of Proposition \ref{prop:Neumann_counting function_upper_bound}.

\begin{prop}\label{prop:Dir_counting_lower}
For any $n\in \N$, $\Omega \in \O_{d}$ and $\alpha > 0$,
    \begin{equation}
    \mathcal{N}_{\Omega}^{D}(\alpha) \geq \frac{n|\Omega|}{(\lambda_{n}^{*}+n^{-1})^{d/2}} \alpha^{d/2} - \frac{2nd^{1/2}|\partial \Omega|}{(\lambda_{n}^{*}+n^{-1})^{(d-1)/2}}\alpha^{(d-1)/2},
\end{equation} where $\lambda_{n}^{*}$ is the $n$-th Dirichlet eigenvalue of the $d$-dimensional unit cube.
\end{prop}

\begin{proof}
    For $m\in \Z^{d}$ and $\delta > 0$, define $Q_{m,\delta}$ and $\mathcal{I}_{\delta}$ as in the proof of Proposition \ref{prop:Neumann_counting function_upper_bound}. It is clear that for a given $m\in \Z^{d}$ if $Q_{m,\delta} \cap \lbrace x \in \Omega : d(x,\partial \Omega) \geq 2\delta d^{1/2} \rbrace \neq \emptyset$ then $m\in \mathcal{I}_{\delta}$. Hence, we obtain that \begin{equation}
        \# \mathcal{I}_{\delta} \geq \delta^{-d}| \lbrace x \in \Omega : d(x,\partial \Omega) \geq 2\delta d^{1/2} \rbrace| \geq \delta^{-d} \Omega - 2d^{1/2} \delta^{-d+1}|\partial \Omega|. 
    \end{equation} Noting that $\mathcal{N}_{Q_{m,\delta}}^{D}(\delta^{-2}(\lambda_{n}^{*}+n^{-1}))\geq n$, by the variational characterisation of Dirichlet eigenvalues, see \cite[Thm. 3.1.9.]{Levitin-Polterovich-Mangoubi-2023}, it suffices to bound the counting function of $\bigcup_{m\in \mathcal{I}_{\delta}} Q_{m,\delta}$ from below. Hence, taking $\delta= \alpha^{-1/2}(\lambda_{n}^{*}+n^{-1})^{1/2}$, we see that \begin{equation}
        \mathcal{N}_{\Omega}^{D}(\alpha) \geq \sum_{m\in \mathcal{I}_{\delta}} \mathcal{N}_{Q_{m,\delta}}^{D}(\alpha) \geq \frac{n|\Omega|}{(\lambda_{n}^{*}+n^{-1})^{d/2}} \alpha^{d/2} - \frac{2nd^{1/2}|\partial \Omega|}{(\lambda_{n}^{*}+n^{-1})^{(d-1)/2}}\alpha^{(d-1)/2},
    \end{equation} which completes the proof.
\end{proof}

\section{Proofs of Theorems \ref{thm:neumann_perim}, \ref{thm:neumann_diam} and  \ref{thm:neumann_per_diam}}\label{sec:main_proofs}

For $d\geq 3$, one can make the upper bound in Proposition \ref{prop:Neumann_counting function_upper_bound} uniform over a given collection of convex domains by constraining the diameter of the domains. This is due to the monotonicity of quermassintegrals as any convex domain of diameter $D>0$ is necessarily contained in the ball of diameter $2D$. Diameter is a very natural quantity to constrain in order to uniformly control Neumann eigenvalues. This is evident from the Payne-Weinberger inequality for $\mu_{2}$. Moreover, upper bounds for Neumann eigenvalues in terms of diameter have been deduced in \cite{Henrot-Michetti-2023} and \cite{Kroger-1999} and lower bounds for Neumann eigenvalues in terms of diameter are known, see Theorem 5.4 and Remark 5.7 in \cite{Cavallani2023}.

We now show how one can construct asymptotic uniform lower bounds on Neumann eigenvalues of convex domains. In fact, we do not need uniform control on the diameter of the convex domains, we only need a certain control for each $k$, as we now prove.

Recall that $\mathcal{O}_{d}$ denotes the collection of bounded convex domains endowed with the Hausdorff topology induced by the metric given by \eqref{eq:Hausdorff_metric} and $W_{d}$ denotes the Weyl constant from \eqref{eq:Weyl_law_orig}.

\begin{prop}\label{prop:Neu_lower_bound} For any $V >0 $ and any $f : \N \to \R_{>0}$ such that $c\leq f(k) = o(k^{1/d(d-1)})$ as $k\to +\infty$ for some $c > 0$,
 \begin{equation}
        \liminf_{k\to +\infty} k^{-2/d} \Big[ \inf \left\lbrace \mu_{k}(\Omega) : \Omega \in \O_{d}, \, |\Omega| \leq V, \, \mathrm{diam}(\Omega) \leq f(k) \right\rbrace \Big] \geq W_{d}V^{-2/d}.
    \end{equation}
\end{prop}

\begin{proof}
Let $k,n\in \mathbb{N}$ and $\epsilon > 0$ be fixed, and let $\Omega\in \mathcal{O}_{d}$ with $|\Omega|\leq V$ and $\diam(\Omega) \leq f(k)$. From the bound in Proposition \ref{prop:Neumann_counting function_upper_bound}, using the monotonicity of the remainder, we see that \begin{equation}
    k \leq \mathcal{N}_{\Omega}^{N}(\mu_{k}(\Omega)+\epsilon) \leq \frac{n|\Omega|}{(\mu_{n+1}^{*})^{d/2}}(\mu_{k}(\Omega)+\epsilon)^{d/2} + r_{n}(B_{k};\mu_{k}(\Omega)+\epsilon),
\end{equation} where $B_{k}$ is the ball of diameter $2f(k)$. Since $\Omega$ was arbitrary, we see that \begin{equation}
    1 \leq k^{-1}\mathcal{N}_{\Omega}^{N}(m_{k}+\epsilon) \leq \frac{nV}{k(\mu_{n+1}^{*})^{d/2}}(m_{k}+\epsilon)^{d/2} + k^{-1}r_{n}(B_{k};m_{k}+\epsilon),
\end{equation} where \begin{equation}
    m_{k} = \inf \left\lbrace \mu_{k}(\Omega) : \Omega \in \O_{d}, \, |\Omega| \leq V, \, \mathrm{diam}(\Omega) \leq f(k) \right\rbrace.
\end{equation} Setting $\overline{m}_{k} = k^{-2/d}(m_{k}+\epsilon)$, writing out the bound we see that \begin{equation}
\begin{split}
   1 & \leq \frac{nV}{(\mu_{n+1}^{*})^{d/2}}(\overline{m}_{k})^{d/2} + C_{d,n} k^{-1/d}|\partial B_{k}|(\overline{m}_{k})^{(d-1)/2} \\
   & \qquad + \sum_{j=2}^{d-1} C_{d,n,j}' s_{j}(B_{k}) k^{-j/d}(\overline{m}_{k})^{(d-j)/2} + C_{d}'' k^{-1} \\
   & := p_{n,k}(\overline{m}_{k}),
   \end{split}
\end{equation} for some constants $C_{d,n},C_{d,n,j}',C_{d}''>0$ whose dependence is denoted in the subscript. By the scaling properties of quermassintegrals we see that $k^{-1/d}|\partial B_{k}|,k^{-j/d}s_{j}(B_{k}) \to 0$ as $k\to +\infty$. Hence for any $0<\delta<1$, for $k$ sufficiently large we have that \begin{equation}
    1 \leq p_{n,k}(\overline{m}_{k}) \leq \frac{nV}{(\mu_{n+1}^{*})^{d/2}}(\overline{m}_{k})^{d/2} + \delta \sum_{j=1}^{d} (\overline{m}_{k})^{(d-j)/2}.
\end{equation} Let $\gamma_{n,\delta}$ be the unique positive solution to \begin{equation}
    \frac{nV}{(\mu_{n+1}^{*})^{d/2}}(\gamma_{n,\delta})^{d/2} + \delta \sum_{j=1}^{d} (\gamma_{n,\delta})^{(d-j)/2} = 1,
\end{equation} then we immediately deduce that $\overline{m}_{k} \geq \gamma_{n,\delta}$. Since $\delta > 0$ was arbitrary, we see that \begin{equation}
    \overline{m}_{k} \geq \lim_{\delta \downarrow 0} \gamma_{n,\delta} = \left(\frac{(\mu_{n+1}^{*})^{d/2}}{nV}\right)^{2/d} = \frac{\mu_{n+1}^{*}}{n^{2/d}V^{2/d}}.
\end{equation} And so \begin{equation}
    \liminf_{k\to +\infty} k^{-2/d}m_{k} \geq \frac{\mu_{n+1}^{*}}{n^{2/d}V^{2/d}}
\end{equation} as $\epsilon > 0$ was arbitrary. Taking the limit as $n\to +\infty$ gives the result by Weyl's law.
\end{proof}

One can also prove the following result completely analagously to that of Proposition \ref{prop:Neu_lower_bound}.

\begin{prop}\label{prop:Neu_lower_bound_2d}
    For any $V >0 $ and any $f : \N \to \R_{>0}$ such that $c\leq f(k) = o(k^{1/2})$ as $k\to +\infty$ for some $c > 0$,
 \begin{equation}
        \liminf_{k\to +\infty} k^{-1} \Big[ \inf \left\lbrace \mu_{k}(\Omega) : \Omega \in \O_{2}, \, |\Omega| \leq V, \, |\partial \Omega| \leq f(k) \right\rbrace \Big] \geq 4\pi V^{-1}.
    \end{equation}
\end{prop}

In two dimensions, we are able to establish the existence of minimisers for all $k\geq 3$ to \eqref{prob:neu_perim} and \eqref{prob:neu_diam} from a simple adaptation of the methods in \cite[\S 2]{Cavallani2023}. We now give a variation of Theorem 2.4. in \cite{Cavallani2023} which is sufficient for our purposes.

\begin{lemma}\label{lem:2d_existence_lemma}
    When $d=2$: \begin{itemize}
        \item[(i)] there exists a minimiser to \eqref{prob:neu_perim} if and only if there exists a domain $\Omega \in \mathcal{O}_{2}$ with $|\partial \Omega| = 1$ such that \begin{equation}
            \mu_{k}(\Omega) \leq 4\pi^{2}(k-1)^{2};
        \end{equation}
        \item[(ii)] there exists a minimiser to \eqref{prob:neu_diam} if and only if there exists a domain $\Omega \in \mathcal{O}_{2}$ with $\diam(\Omega) = 1$ such that \begin{equation}
            \mu_{k}(\Omega) \leq \pi^{2}(k-1)^{2}.
            \end{equation}
    \end{itemize}
\end{lemma}

\begin{proof}
    The proof of both (i) and (ii) can be deduced immediately following the lines of the proof of Theorem 2.4. in \cite{Cavallani2023}, noting that in the case (i) if a sequence of unit perimeter planar convex sets collapses it necessarily must collapse to a line segment of length $1/2$ and in case (ii) if a sequence of unit diameter planar convex sets collapses it necessarily must collapse to a line segment of length $1$.
\end{proof}

The proofs of Theorems \ref{thm:neumann_perim}, \ref{thm:neumann_diam} and \ref{thm:neumann_per_diam} now immediately follow from the proof of Propositions \ref{prop:Neu_lower_bound} and \ref{prop:Neu_lower_bound_2d}.

\begin{proof}[{Proof of Theorem \ref{thm:neumann_perim}}]
In light of Lemma \ref{lem:2d_existence_lemma}(i), it suffices to show that there exists a domain $\Omega \in \mathcal{O}_{2}$ such that $\mu_{k}(\Omega) \leq 4\pi^{2}(k-1)^{2}$ for all $k\geq 3$ for the existence part of Theorem \ref{thm:neumann_perim}. The Neumann eigenvalues of the square $Q$ unit perimeter are well-known to be given by $16\pi^{2}\left[(i-1)^{2}+(j-1)^{2}\right]$, $i,j\in \mathbb{N}$, and moreover it is known that the Neumann eigenvalues of $Q$ satisfy the Pólya inequality $\mu_{k}(Q) \leq 64\pi (k-1)$. As the inequality $64\pi (k-1) \leq 4\pi^{2}(k-1)^{2}$ holds for $k\geq 7$, we know that there exists a minimiser to \eqref{prob:neu_perim} when $d=2$ for all $k\geq 7$. We do the cases $k=2,\ldots,6$ explicitly. The first six Neumann eigenvalues of $Q$, counting multiplicities, are given in ascending order by $0$, $16\pi^{2}$, $16\pi^{2}$, $32\pi^{2}$, $64\pi^{2}$ and $64\pi^{2}$. By direct comparison using Lemma \ref{lem:2d_existence_lemma}(i) we see that we must have existence of a minimiser for all $k\geq 3$ but not when $k=2$.

Now we know the existence of minimisers $\Omega_{k}^{*}$ to \eqref{prob:neu_perim}, when $d=2$, we need to prove the asymptotic behaviour of any sequence $(\Omega_{k}^{*})_{k\geq 1}$ of minimisers as $k\to +\infty$. Let $f(k)=1$ and take $V > 0$ to be the volume of the two-dimensional ball of unit perimeter, which we denote by $B$. Using Weyl's law and Proposition \ref{prop:Neu_lower_bound_2d}, we see that for any $\epsilon>0$ for $k$ sufficiently large \begin{equation}
        \mu_{k}(B) < \inf \left\lbrace \mu_{k}(\Omega) : \Omega \in \O_{2}, \, |\Omega| \leq V-\epsilon, \, |\partial \Omega| = 1 \right\rbrace.
    \end{equation} Hence, one must have that $|\Omega_{k}^{*}| > V-\epsilon$ for $k$ sufficiently large. Since $\epsilon > 0$ was arbitrary, we see that $|\Omega_{k}^{*}| \to V$ as $k\to +\infty$. Using Bonnesen's quantitative isoperimetric inequality, \cite{Bonnesen-1924,Kritikos-1927}, one can deduce that the $\partial\Omega_{k}^{*}$ eventually lie, up to rigid planar motions, inside the annulus \begin{equation}
        (\partial B)_{\delta} := \lbrace x \in \mathbb{R}^{2} : d(x,\partial B) \leq \delta \rbrace
    \end{equation} for any $\delta>0$ for $k$ sufficiently large. Hence, the $\Omega_{k}^{*}$ Hausdorff converge, up to possible rigid planar motions, to $B$ as $k\to +\infty$, which completes the proof.

\end{proof}

\begin{proof}[{Proof of Theorem \ref{thm:neumann_diam}}]

In dimension two, the proof of existence of minimisers for all $k\geq 3$ and not for $k=2$ follows completely analogously as in the proof of Theorem \ref{thm:neumann_diam}. In higher dimensions, we cannot use the same trick as in two-dimensions as collapsing sequences of convex domains of unit diameter do not necessarily converge to a line segment. Instead we show that minimisers must eventually exist from the asymptotic result in Proposition \ref{prop:Neu_lower_bound}.

Let $f(k)=1$ and take $V > 0$ to be the volume of the $d$-dimensional ball of unit diameter, which we denote by $B$. Using Weyl's law and Proposition \ref{prop:Neu_lower_bound}, we see that for any $\epsilon>0$ for $k$ sufficiently large \begin{equation} \label{eq:ball_comparison}
    \mu_{k}(B) < \inf \left\lbrace \mu_{k}(\Omega) : \Omega \in \O_{d}, \, |\Omega| \leq V - \epsilon, \, \mathrm{diam}(\Omega) \leq 1 \right\rbrace.
\end{equation} Hence, \begin{equation}\begin{split}\label{eq:diam_compactness}
   &  \inf \left\lbrace \mu_{k}(\Omega) : \Omega \in \O_{d}, \, \mathrm{diam}(\Omega)
     = 1 \right\rbrace \\ & \qquad \qquad  = \inf \left\lbrace \mu_{k}(\Omega) : \Omega \in \O_{d}, \, |\Omega| \geq V -\epsilon, \, \mathrm{diam}(\Omega) = 1 \right\rbrace.
    \end{split} 
\end{equation} By the continuity of volume and diameter in the Hausdorff topology and Blaschke's selection theorem \cite[Thm 6.3.]{Gruber-2007}\footnote{On a technical note, Blaschke's selection theorem is stated for compact convex sets with possibly zero volume. The convex domains in \eqref{eq:diam_compactness} have their inradii uniformly bounded from below and hence we avoid having this possibility.}, infimum on the RHS of \eqref{eq:diam_compactness} is taken over a set which is sequentially compact. Moreover, Neumann eigenvalues are continuous with respect to Hausdorff convergence of convex domains, see for example \cite{Ross-2004}, and so a simple application of the extreme value theorem shows that a minimiser must necessarily exist for $k$ sufficiently large.

As $\epsilon > 0$ was arbitrary in \eqref{eq:ball_comparison}, it is clear that one necessarily have that for any sequence $(\Omega_{k}^{*})_{k}$ of minimisers that $|\Omega_{k}^{*}|\to V$ as $k\to +\infty$. Using the quantitative isodiametric inequality \cite[Thm. 1]{Maggi-Ponsiglione-Pratelly-2014}, one can deduce that the $\Omega_{k}^{*}$ necessarily Hausdorff converge, up to possible rigid planar motions, to $B$ as $k\to +\infty$, which completes the proof. 

One can argue alternatively by contradiction by supposing that a subsequence of $\Omega_{k}^{*}$, also denoted $\Omega_{k}^{*}$, does not converge to the ball $B$ and then one would have that $\limsup_{k\to +\infty}|\Omega_{k}^{*}| < V$. But then by Proposition \ref{prop:Neu_lower_bound} \begin{equation}
        \liminf_{k\to +\infty} \frac{\mu_{k}(\Omega_{k}^{*})}{k^{2/d}} \geq \frac{W_{d}}{\displaystyle \left(\limsup_{k\to+\infty} |\Omega_{k}^{*}|\right)^{2/d}} \geq \frac{W_{d}}{V^{2/d}} = \lim_{k\to +\infty} \frac{\mu_{k}(B)}{k^{2/d}},
    \end{equation} which contradicts the optimality of the $\Omega_{k}^{*}$.
\end{proof}

Following the same lines of argument of the proof of Theorem \ref{thm:neumann_diam}, and using the quantitative isoperimetric inequality results due to Fuglede in \cite{Fuglede-1989}, one can prove Theorem \ref{thm:neumann_per_diam} in the same way.

\section{Geometric stability of Weyl's law}\label{sec:other_applications}

In this section, we discuss further applications of the bounds obtained in Section \ref{sec:counting_bounds} to stability results concerning Weyl's law, see Theorem \ref{thm:generalised_Weyl_law}. These results will become of use when one wants to obtain asymptotic shape optimisation results for mixed Dirichlet-Neumann eigenvalues in Section \ref{sec:zaremba}.

Recall again that $\mathcal{O}_{d}$ denotes the collection of bounded convex domains endowed with the Hausdorff topology induced by the metric given by \eqref{eq:Hausdorff_metric} and $W_{d}$ denotes the Weyl constant from \eqref{eq:Weyl_law_orig}.

Firstly, we can carry out similar reasoning to that in the last section to yield asymptotic upper bounds for Dirichlet eigenvalues.

\begin{prop}\label{prop:dir_upper_bound} For any $V >0 $ and any $f : \N \to \R_{>0}$ such that $f(k) = o(k^{1/d})$ as $k\to +\infty$,
 \begin{equation}
        \limsup_{k\to +\infty} k^{-2/d} \Big[ \sup \left\lbrace \lambda_{k}(\Omega) : \Omega \in \O_{d}, \, |\Omega| \geq V, \, |\partial \Omega| \leq f(k) \right\rbrace \Big] \leq W_{d}V^{-2/d}
    \end{equation} as $k\to +\infty$, provided that the set \begin{equation}
        \left\lbrace \lambda_{k}(\Omega) : \Omega \in \O_{d}, \, |\Omega| \geq V, \, |\partial \Omega| \leq f(k) \right\rbrace
    \end{equation} remains non-empty.
\end{prop}

\begin{proof}
    For an arbitrary $\Omega \in \O_{d}$ with $|\Omega| \geq V$ and $|\partial \Omega| \leq f(k)$, observe that \begin{equation}
    \begin{split}
        k \geq \mathcal{N}_{\Omega}^{D}(\lambda_{k}(\Omega)) & \geq \frac{n|\Omega|}{(\lambda_{n}^{*}+n^{-1})^{d/2}} \lambda_{k}(\Omega)^{d/2} - \frac{2nd^{1/2}|\partial \Omega|}{(\lambda_{n}^{*}+n^{-1})^{(d-1)/2}}\lambda_{k}(\Omega)^{(d-1)/2} \\
        & \geq \frac{nV}{(\lambda_{n}^{*}+n^{-1})^{d/2}} \lambda_{k}(\Omega)^{d/2} - \frac{2nd^{1/2}f(k)}{(\lambda_{n}^{*}+n^{-1})^{(d-1)/2}}\lambda_{k}(\Omega)^{(d-1)/2},
    \end{split}
    \end{equation} using the bound from Proposition \ref{prop:Dir_counting_lower}. Since $\Omega$ was chosen arbitrarily, we have that \begin{equation}\label{eq:dir_up_bound_line}
        1 \geq  \frac{nV}{(\lambda_{n}^{*}+n^{-1})^{d/2}} \left(\frac{M_{k}}{k^{2/d}}\right)^{d/2} - \frac{2nd^{1/2}f(k)}{k^{1/d}(\lambda_{n}^{*}+n^{-1})^{(d-1)/2}}\left(\frac{M_{k}}{k^{2/d}}\right)^{(d-1)/2},
    \end{equation} where \begin{equation}
        M_{k} = \sup \left\lbrace \lambda_{k}(\Omega) : \Omega \in \O_{d}, \, |\Omega| \geq V, \, |\partial \Omega| \leq f(k) \right\rbrace.
    \end{equation} Since $f(k) = o(k^{1/d})$, we see that $k^{-1/d}f(k)$ is a bounded sequence and so taking $n=1$, we have that there exist constants $C_{1},C_{2}>0$ such that \begin{equation}
        C_{1}\left(\frac{M_{k}}{k^{2/d}}\right)^{d/2}-C_{2}\left(\frac{M_{k}}{k^{2/d}}\right)^{(d-1)/2} - 1 \leq 0.
    \end{equation} From which it immediately follows there exists a constant $C>0$ such that \begin{equation}
        M_{k} \leq C k^{2/d}.
    \end{equation} Now in view of equation \eqref{eq:dir_up_bound_line}, we have \begin{equation}
        1 \geq  \frac{nV}{(\lambda_{n}^{*}+n^{-1})^{d/2}} \left(\frac{M_{k}}{k^{2/d}}\right)^{d/2} - \frac{2nd^{1/2}f(k)}{k^{1/d}(\lambda_{n}^{*}+n^{-1})^{(d-1)/2}}C^{(d-1)/2},
    \end{equation}
Taking the limsup as $k\to +\infty$ and rearranging yields that \begin{equation}
        \limsup_{k\to+\infty} k^{-2/d}M_{k} \leq \frac{\lambda_{n}^{*}+n^{-1}}{n^{2/d}V^{2/d}}. 
    \end{equation} Since $n\in \N$ was arbitrary we see that \begin{equation}
        \limsup_{k\to+\infty} k^{-2/d}M_{k} \leq W_{d}V^{-2/d}
    \end{equation} using Weyl's law, which completes the proof.
\end{proof}

As a direct consequence of Propositions \ref{prop:Neu_lower_bound} and \ref{prop:dir_upper_bound} one can deduce the following `generalisation' of Weyl's law for bounded convex domains.
\begin{thm}\label{thm:generalised_Weyl_law}
    Let $\Omega_{k}\subset \mathbb{R}^{d}$ be a sequence of bounded convex domains of volume $V > 0$ and $\diam(\Omega_{k}) = o(k^{1/(d(d-1))})$, then \begin{equation}\label{eq:Weyl_law}
        \lambda_{k}(\Omega_{k}) \sim \mu_{k}(\Omega_{k}) \sim 4\pi^{2} \left(\frac{k}{\omega_{d} V}\right)^{2/d}
    \end{equation} as $k\to +\infty$.
\end{thm}

\begin{proof}
    Noting that the condition $\diam(\Omega_{k}) = o(k^{1/(d(d-1))})$ implies that $|\partial \Omega_{k}|=o(k^{1/d})$ and that by classical variational arguments $\mu_{k}(\Omega_{k}) \leq \lambda_{k}(\Omega_{k})$, combining a simple application of Propositions \ref{prop:Neu_lower_bound} and \ref{prop:dir_upper_bound} gives the result.
\end{proof}

The condition $\diam(\Omega_{k}) = o(k^{1/(d(d-1))})$ is sharp in the sense that one can construct sequences of domains with $\diam(\Omega_{k}) = O(k^{1/(d(d-1))})$ for which \eqref{eq:Weyl_law} does not hold. For example, in two dimensions one can consider the sequence of domains $\Omega_{k} = (0,(4k)^{1/2}) \times (0,(4k)^{-1/2})$. The philosophy of Theorem \ref{thm:generalised_Weyl_law} is that if we do allow the geometry of the $\Omega_{k}$ to degenerate too quickly as $k\to +\infty$ then the leading Weyl term will dominate against any remainder. This, at least heuristically, explains why Theorem \ref{thm:neumann_per_diam} holds true. We will use Theorem \ref{thm:generalised_Weyl_law} in the proof of Theorem \ref{thm:Zaremba_surface_measure} in the next section to illustrate how it can be used.

\section{Mixed boundary conditions}\label{sec:zaremba}

As discussed in the introduction, the minimisation problem \eqref{prob:dir_perim} for Dirichlet eigenvalues under perimeter constraint is well-posed but the minimisation problem \eqref{prob:neu_perim} for Neumann eigenvalues under perimeter constraint is ill-posed. Here we consider a non-trivial minimisation problem for eigenvalues of the Laplacian under mixed Dirichlet-Neumann, so-called Zaremba, boundary conditions under perimeter constraint which is well-posed and has the same asymptotic behaviour as in Theorem \ref{thm:Dirichlet_perim}.

To do this we define a subcollection $\mathcal{O}_{d,L}$ of $\mathcal{O}_{d}$ such that for each domain in the collection there is a canonical way of prescribing the mixed boundary conditions and the minimisation problem itself is well-posed. The definition of this subcollection is subtle and may appear somewhat odd at first but it allows us to obtain uniform lower bounds on the Zaremba eigenvalues and deduce the continuity of the Zaremba eigenvalues over $\mathcal{O}_{d,L}$. Without further ado, we give the definition of this subcollection below.

Let $\wp$ be the canonical projection $\mathbb{R}^{d} \to \mathbb{R}^{d-1}$ which omits the final coordinate. $\wp$ will denote this projection throughout the rest of the paper. Given $\Omega \in \mathcal{O}_{d}$, its image under $\wp$, denoted $\wp(\Omega)$, is a convex domain in $\mathbb{R}^{d-1}$. For each $x' \in \wp(\Omega)$ we can define two functions $h^{+},h^{-}:\wp(\Omega) \to \mathbb{R}$ by \begin{equation}
    h^{+}(x') = \sup\lbrace y \in \mathbb{R} : (x',y) \in \Omega \rbrace, \quad h^{-}(x') = \inf\lbrace y \in \mathbb{R} : (x',y) \in \Omega \rbrace.
\end{equation} We call $h^{+}$ and $h^{-}$ the upper and lower profiles of $\Omega$ and as functions they are concave and convex respectively. These functions are well defined as any line passing through a convex domain intersects the boundary precisely twice. Given $L>0$, we say that $\Omega$ is a convex $L$-Lip domain if $h^{+}$ and $h^{-}$ are both $L$-Lipschitz and agree on the boundary of $\wp(\Omega)$, denoted $\partial \wp(\Omega)$. We denote the collection of all convex $L$-Lip domains in $\mathbb{R}^{d}$ by $\mathcal{O}_{d,L}$. We define the upper boundary of $\Omega$ by $\Gamma^{+}:= \Gamma^{+}(\Omega) := \lbrace (x',h^{+}(x')) : x' \in \wp(\Omega) \rbrace \subset \partial \Omega$ and define the lower boundary $\Gamma^{-} := \Gamma^{-}(\Omega)$ analogously.

Let $\Omega \in \O_{d,L}$. We define the Zaremba Sobolev space $\mathcal{H}_{0,\Gamma^{-}}^{1}(\Omega)$ as the completion of the space \begin{equation}
    C_{0,\Gamma^{-}}^{\infty}(\Omega) = \lbrace  \left.\phi \right|_{\Omega} \in C^{\infty}(\Omega): \phi \in C_{0}^{\infty}(\mathbb{R}^{d}), \,  d(\mathrm{supp}(\phi), \Gamma^{-}) > 0\rbrace
\end{equation} in the Sobolev norm \begin{equation}
    \lVert u\rVert_{\mathcal{H}^{1}} := \left(\int_{\Omega} |\nabla u|^{2}  + \int_{\Omega}u^{2}\right)^{1/2}.
\end{equation} Then, in the usual way, we define the Zaremba Laplacian $-\Delta_{\Omega}^{Z}$ on $\L^{2}(\Omega)$ via the Friedrich's extension with domain \begin{equation}
    \mathrm{dom}(-\Delta_{\Omega}^{Z}) = \left\lbrace u \in \mathcal{H}_{0,\Gamma^{-}}^{1}(\Omega) : \Delta u \in \L^{2}(\Omega), \, \left.\partial_{n}u\right|_{\Gamma^{+}} = 0 \right\rbrace,
\end{equation} where the conditions in the definition of $\mathrm{dom}(-\Delta_{\Omega}^{Z})$ are understood in the distributional sense. The Zaremba Laplacian $-\Delta_{\Omega}^{Z}$ has the associated non-negative symmetric quadratic form \begin{equation}
    Q(u,v) = \int_{\Omega} \nabla u \cdot \nabla v
\end{equation} with domain $\mathrm{dom}(Q) = \mathcal{H}_{0,\Gamma^{-}}^{1}(\Omega)$ and one can deduce that $-\Delta_{\Omega}^{Z}$ has a discrete collection of positive eigenvalues accumulating only at $+\infty$, which we denote \begin{equation}
    0 < \zeta_{1}(\Omega) < \zeta_{2}(\Omega) \leq \zeta_{3}(\Omega) \leq \cdots \uparrow +\infty,
\end{equation} that have the variational characterisation \begin{equation}\label{eq:zar_var_char}
     \zeta_{k}(\Omega) = \min_{\substack{S \subseteq \mathcal{H}_{0,\Gamma^{-}}^{1}(\Omega) \\ \dim (S) = k}} \max_{\substack{u \in S \\ u\neq 0}} \frac{\int_{\Omega} |\nabla u|^{2}}{\int_{\Omega} u^{2}}.
\end{equation}

For our purposes we only need the definition of $\mathcal{H}_{0,\Gamma^{-}}^{1}(\Omega)$ and the variational characterisation given in \eqref{eq:zar_var_char}. For a fuller discussion on defining Zaremba eigenvalues we direct the reader to \cite[\S 2]{Lotoreichik-Rohleder-2017} and \cite[\S 3.1.3]{Levitin-Polterovich-Mangoubi-2023}, and the references therein.

Now that $\O_{d,L}$ has been defined and we have defined Zaremba eigenvalues on domains lying in $\mathcal{O}_{d,L}$ we are ready to state our main results.
\begin{thm}\label{thm:Zaremba_surface_measure}
    For any $d\geq 2$ and $L>0$, for all $k\geq 1$ there exists a minimiser $\Omega_{k}^{*}$ to the problem \begin{equation}\label{eq:Zaremba_perim_problem}
        \inf \lbrace \zeta_{k}(\Omega) : \Omega \in \mathcal{O}_{d,L}, \, |\partial \Omega| = 1\rbrace.
    \end{equation}
    Moreover, any sequence $\Omega_{k}^{*}$ of minimisers is non-degenerate, i.e. $\liminf_{k\to +\infty}|\Omega_{k}^{*}|>0$, and any accumulation point, up to possible rigid planar motions, of $\Omega_{k}^{*}$ is a solution to the isoperimetric problem over $\mathcal{O}_{d,L}$, which is necessarily symmetric, up to a translation, about the hyperplane $\lbrace x_{d} = 0\rbrace$.
\end{thm}

We note that for any $k\geq1$ and $d\geq 3$, \begin{equation}\label{eq:zar_inf_zero}
    \inf \left\lbrace \zeta_{k}(\Omega) : \Omega \in \bigcup_{L>0} \mathcal{O}_{d,L}, \, |\partial \Omega| = 1\right\rbrace = 0
\end{equation} and so without a uniform $L$-Lipschitz constraint the conclusion of Theorem \ref{thm:Zaremba_surface_measure} fails to hold. Moreover, for any $k\geq 1$ and $d\geq 3$,  \begin{equation}\label{eq:neu_inf_zero}
    \inf \left\lbrace \mu_{k}(\Omega) : \Omega \in  \mathcal{O}_{d,L}, \, |\partial \Omega| = 1\right\rbrace = 0
\end{equation} for all $k\in \N$ and so the Zaremba eigenvalues behave fundamentally differently to Neumann eigenvalues over the collection $\O_{d,L}$. 

We now briefly illustrate \eqref{eq:zar_inf_zero} and \eqref{eq:neu_inf_zero} as an example when $d=3$, the higher dimensional cases can be done similarly. 

\begin{ex}
    Let $0 < \epsilon < 1$ and set $R_{\epsilon} = (0,\epsilon) \times (0,\epsilon)$. Define $f_{L,\epsilon} : R_{\epsilon} \to \mathbb{R}$ by $f_{L,\epsilon}(x,y) = \min\lbrace L d((x,y), \partial R_{\epsilon}), \epsilon^{-1}\rbrace$ and let \begin{equation}
        \Omega_{L,\epsilon} := \lbrace (x,y,z) : (x,y)\in R_{\epsilon}, \, 0 < z < f_{L,\epsilon}(x,y) \rbrace,
    \end{equation} which lies in $\mathcal{O}_{3,L}$. Let $u_{j}(x,y,z)= \sin(\pi (j+1/2) \epsilon z)$ for $1 \leq j \leq k$, then we see that \begin{equation}
        \zeta_{k}(\Omega_{L,\epsilon}) \leq \max_{(\alpha_{1},\ldots,\alpha_{k}) \neq (0,\ldots,0)} \frac{\displaystyle \int_{\Omega_{L,\epsilon}} \left|\nabla \sum_{j=1}^{k} \alpha_{j} u_{j} \right|^{2}}{\displaystyle \int_{\Omega_{L,\epsilon}} \left| \sum_{j=1}^{k} \alpha_{j} u_{j} \right|^{2}} \to \pi^{2}(k+1/2)^{2}\epsilon^{2}
    \end{equation} as $L\to +\infty$. Moreover, $|\partial \Omega_{L,\epsilon}| \to 4+2\epsilon^{2}$ as $L\to +\infty$. Since $0 < \epsilon < 1$ was arbitrary, by the properties of Zaremba eigenvalues under scaling, we see that \eqref{eq:zar_inf_zero} indeed holds for any $k\geq 1$ when $d=3$.

    Now set $S_{\epsilon} := (0,\epsilon^{-1})\times (0,\epsilon)$ and set $g_{L,\epsilon}: S_{\epsilon} \to (0,+\infty)$ by $g_{L,\epsilon}(x,y) := L d((x,y),\partial S_{\epsilon})$ and define the domain $D_{\epsilon}$ by \begin{equation}
        D_{\epsilon} := \lbrace (x,y,z) : (x,y) \in R_{\epsilon}, \, -g_{L,\epsilon}(x,y) < z < g_{L,\epsilon}(x,y) \rbrace. 
    \end{equation} We have that $D_{\epsilon} \in \mathcal{O}_{3,L}$ and that $|\partial \Omega_{\epsilon}| = 2\sqrt{1+L^{2}}$ for all $0 < \epsilon < 1$. Letting $u_{j}(x,y,z) = \cos(\pi j \epsilon x)$ for $1\leq j \leq k$, one sees that \begin{equation}
    \begin{split}
        \mu_{k}(\Omega_{\epsilon}) & \leq \max_{\substack{v = \alpha_{1}u_{1} + \cdots +\alpha_{k}u_{k} \\ \lVert v\rVert_{\mathcal{L}^{2}} = 1}} \int_{\Omega_{\epsilon}} |\nabla v|^{2} \\
        &  = \max_{\substack{v = \alpha_{1}u_{1} + \cdots +\alpha_{k}u_{k} \\ \lVert v\rVert_{\mathcal{L}^{2}} = 1}} \int_{0}^{\epsilon^{-1}} dx \, \int_{0}^{\epsilon} dy \, \int_{-g_{\epsilon}(x,y)}^{g_{\epsilon}(x,y)} dz\, \left| \nabla \sum_{j} \alpha_{j} u_{j}(x,y,z)\right|^{2} \\
        & \leq L\epsilon^{2} \max_{\substack{v = \alpha_{1}u_{1} + \cdots +\alpha_{k}u_{k} \\ \lVert v\rVert_{\mathcal{L}^{2}} = 1}} \int_{0}^{\epsilon^{-1}} dx \, \left| \nabla \sum_{j} \alpha_{j} u_{j}(x,y,z)\right|^{2} \\
        & =  L\pi^{2}k^{2} \epsilon^{4}.
    \end{split}
    \end{equation} Since $0 < \epsilon < 1$ was arbitrary, by the scaling properties of Neumann eigenvalues under homothety, we see that \eqref{eq:neu_inf_zero} indeed holds for any $k\geq 1$ when $d=3$.
\end{ex}

In the same way as one proves Theorem \ref{thm:Zaremba_surface_measure}, one can also deduce the analogous result in the case of diameter constraint.

\begin{thm}\label{thm:Zaremba_diamater}
    For any $d\geq 2$ and $L>0$, for all $k\geq 1$ there exists a minimiser $\Omega_{k}^{*}$ to the problem \begin{equation}
        \inf \lbrace \zeta_{k}(\Omega) : \Omega \in \mathcal{O}_{d,L}, \, \mathrm{diam}(\Omega) = 1\rbrace.
    \end{equation}
    Moreover, any sequence $\Omega_{k}^{*}$ of minimisers is non-degenerate and any accumulation point, up to possible rigid planar motions, of $\Omega_{k}^{*}$ is a solution to the isodiametric problem over $\mathcal{O}_{d,L}$.
\end{thm}

To make our results clearer, let us illuminate Theorem \ref{thm:Zaremba_surface_measure} through an example in two dimensions.

\begin{ex}
    Fix $0< \delta \leq \frac{\pi}{4}$. Let $\Omega \subset \mathbb{R}^{2}$ be a kite of unit perimeter, let $\ell$ be the line of symmetry of $\Omega$ and assume that the angles that $\ell$ passes through are less than or equal to $\pi-2\delta$, see Figure \ref{fig:kite_example} for an example of this. The collection of such kites is closed in the Hausdorff metric. Partition the boundary of the kite into two disjoint relatively open components $\Gamma^{+}$ and $\Gamma^{-}$ which lie on either side of $\ell$ and, up to a set of measure zero, cover $\partial \Omega$. Then one can define the Zaremba Laplacian for kites in the way described earlier in this subsection. Then Theorem \ref{thm:Zaremba_surface_measure} gives that for sufficiently large $k$ there exists a minimiser $\Omega_{k}^{*}$ of the $k$-th Zaremba eigenvalue among such kites with unit perimeter, and the isoperimetric problem for kites implies that any sequence of such optimisers must converge to the square of unit perimeter as $k\to +\infty$.

    \begin{figure}
        \centering
        \begin{tikzpicture}
        \draw[gray] (-0.75, 0) node[above] {{\tiny $\geq \delta$}};
        \draw[thick,gray] (0.25,0.433) arc (60:90:0.5);
            \draw[thick,gray] (0,-0.5) arc (270:300:0.5);
            \draw[->,gray] (-0.5, 0.25) -- (-0.05, 0.25);
            \draw[->,gray] (-0.5, 0.2) -- (-0.05, -0.25);
        \draw[dashed] (-1,0) -- (5,0) node[right] {$\ell$};
            \draw[red,thick] (0,0) -- (1,1.732) -- (4,0);
            \draw (2.5,1.5) node {$\Gamma^{+}$};
            \draw (2.5,-1.5) node {$\Gamma^{-}$};
            \draw[blue,thick] (0,0) -- (1,-1.732) -- (4,0);
            
            \draw[dashed] (0,-1) -- (0,1);
            
        \end{tikzpicture}
        \caption{An example of symmetric Zaremba boundary conditions on a kite about its axis of symmetry, with Dirchlet boundary conditions denoted in blue and Neumann boundary conditions denoted in red.}
        \label{fig:kite_example}
    \end{figure}
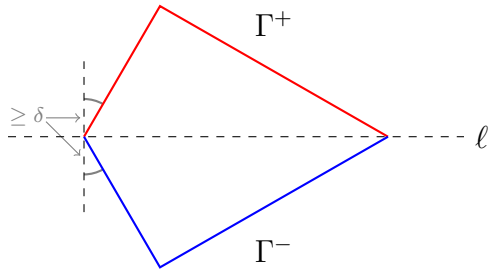

    As a corollary, one can carry out the same for rhombii where $\ell$ is the line of symmetry passing through the smallest opposite pair of interior angles. Then under perimeter constraint, again one has existence of optimisers for $k$ sufficiently large and that the optimisers necessarily converge to the square of unit perimeter as $k\to + \infty$.
\end{ex}

We now turn our attention to proving Theorems \ref{thm:Zaremba_surface_measure} and \ref{thm:Zaremba_diamater}. An easy but key observation to make is that $\mathcal{O}_{d,L}$ is closed under homothety. We begin the section by showing $\O_{d,L}$ is closed in the Hausdorff topology provided that one does not have degeneracy of the volume in the limit. Then we use the definition of $\O_{d,L}$ to prove the continuity of these Zaremba eigenvalues in the Hausdorff topology and then prove a Li-Yau type lower bound for these eigenvalues. Both the proofs of the continuity and the lower bound, require the use of Sobolev extension operators and the choice of definition of $\O_{d,L}$ will become more apparent throughout this section.

\subsection{Properties of $\mathcal{O}_{d,L}$}

\begin{lemma}\label{lem:Odl_closed}
    If $\Omega_{n} \in \mathcal{O}_{d,L}$ is a sequence of domains Hausdorff converging to $\Omega \in \O_{d}$ as $n\to+\infty$, then $\Omega \in \O_{d,L}$.
\end{lemma}

\begin{proof}
    By the invariance of $\mathcal{O}_{d,L}$ under homothety and translations and standard properties of Hausdorff convergence of convex domains, it suffices to prove the result in the case of sequences $\Omega_{n}$ that lie in $\mathcal{O}_{d,L}$ which Hausdorff converge to some bounded convex domain $\Omega$ as $n\to +\infty$ and for which $\Omega_{n} \subset \Omega$ for each $n$. Let $h_{n}^{+}: \wp(\Omega_{n}) \to \mathbb{R}$ be the upper height function of $\Omega_{n}$ and $h^{+}$ the upper height function of $\Omega$. Now fix $x',y'\in \wp(\Omega)$ and let $\epsilon = \frac{1}{2}\min\lbrace d(x',\partial \wp(\Omega)),d(y',\partial \wp(\Omega))\rbrace$. Then it is easy to see that $\wp(\Omega_{n})$ Hausdorff converges to $\wp(\Omega)$ as $n\to +\infty$ so we have that $B(x',\epsilon),B(y',\epsilon)\subset \wp(\Omega_{n})$ for $n$ sufficiently large. Now also for $n$ sufficiently large we see that $d(\partial \Omega, \partial \Omega_{n}) < \epsilon$ by standard results of Hausdorff convergence of convex domains and so, as $\Omega_{n} \subset \Omega$, it is clear that there exist sequences $(x_{n}',h_{n}^{+}(x_{n}')),(y_{n}',h_{n}^{+}(y_{n}')) \in \partial \Omega_{n}$ converging to $(x',h^{+}(x')) (y',h^{+}(y'))\in \partial \Omega$ as $n\to +\infty$. In particular, we see that these sequences can be chosen so that \begin{equation}
        \lVert (x_{n}',h_{n}^{+}(x_{n}'))-(x',h^{+}(x'))\rVert_{2}, \lVert (y_{n}',h_{n}^{+}(y_{n}'))-(y',h^{+}(y'))\rVert_{2} \leq d_{H}(\partial \Omega_{n},\partial \Omega).
    \end{equation} Then \begin{equation}
    \begin{split}
        |h^{+}(x')-h^{+}(y')| & \leq |h^{+}(x')-h_{n}^{+}(x_{n}')| + |h_{n}^{+}(x_{n}')-h_{n}^{+}(y_{n}')| + |h_{n}^{+}(y_{n}')-h^{+}(y')| \\
        & \leq 2d_{H}(\partial \Omega_{n},\partial\Omega) + L|x_{n}'-y_{n}'|.
    \end{split}
    \end{equation} Taking the limit as $n\to +\infty$ we see that $h^{+}$ is $L$-Lipschitz. Similarly one can show that $h^{-}$, the lower height function of $\Omega$, is $L$-Lipschitz. The fact that $h^{+}$ and $h^{-}$ agree on the boundary $\partial \wp(\Omega)$ is easy to argue by contradiction.
\end{proof}

\begin{lemma}\label{lem:Lip_boundary}
    If $\Omega_{n}$ is a sequence of domains in $\mathcal{O}_{d,L}$ Hausdorff converging to a domain $\Omega \in \O_{d,L}$ as $n\to +\infty$, then $\Gamma_{n}^{-}:=\Gamma^{-}(\Omega_{n})$ Hausdorff converges to $\Gamma^{-}:=\Gamma^{-}(\Omega)$ as $n\to+\infty$.
\end{lemma}

\begin{proof}
    As in the proof of Lemma \ref{lem:Odl_closed}, we may assume that $\Omega_{n} \subset \Omega$ for each $n\in \N$. For $\delta > 0$ define the compact subset \begin{equation}
        K_{\delta} := \lbrace (x',y) \in \wp(\Omega)\times \mathbb{R} : h^{-}(x') + \delta \leq y \leq h^{+}(x') - \delta \rbrace \subset \Omega.
    \end{equation} Then for $n$ sufficiently large, we see that $K_{\delta} \subset \Omega_{n}$. Fix $(x',h^{-}(x'))\in \Gamma^{-}(\Omega)$, then let $x_{\delta}'$ be the closest point in $\wp(K_{\delta})$ to $x'$. Then clearly $|x'-x_{\delta}'| \leq d_{H}(K_{\delta},\Omega)$ and so \begin{equation}
    \begin{split}
        |h^{-}(x')-h_{n}^{-}(x_{\delta}')| & \leq |h^{-}(x')-h^{-}(x_{\delta}')| + |h^{-}(x_{\delta}')-h_{n}^{-}(x_{\delta}')|  \\
        & \leq Ld_{H}(K_{\delta},\Omega) + \delta.
    \end{split}
    \end{equation} Since $\delta>0$ was arbitrary we see that $\sup_{x\in \Gamma^{-}(\Omega)}\inf_{y\in\Gamma^{-}(\Omega_{n})} \lVert x-y\rVert_{2} \to 0$ as $n\to +\infty$. One can then deduce that $\sup_{x\in \Gamma^{-}(\Omega_{n})}\inf_{y\in\Gamma^{-}(\Omega)} \lVert x-y\rVert_{2} \to 0$ as $n\to +\infty$ similarly.
\end{proof}

\subsection{Continuity of the $\zeta_{k}$}
We now move on to prove the continuity of these Zaremba eigenvalues in the Hausdorff topology. In \cite{Chenais-1975}, Chenais proved the continuity of solutions to the Neumann problem for domains satisfying a uniform cone condition with respect to the Hausdorff metric. A crucial part of Chenais' proof is to show that over such a collection of domains there exists a uniform constant such that there exists a Sobolev extension operator $\mathcal{H}^{1}(\Omega) \to \mathcal{H}^{1}(\mathbb{R}^{d})$ whose norm is at most this uniform constant. Then from the continuity of the solutions to the Neumann problem, you can prove the continuity of Neumann eigenvalues with respect to the Hausdorff metric, see \cite[\S 3]{Henrot-2006}. The issue that arises in the Zaremba problem is that one wants to extend by zero on the Dirichlet parts of the boundary and extend non-trivially along the Neumann parts of the boundary. This is an inherently tricky situation as you wish to extend by zero near/on Dirichlet parts of the boundary but cannot do so on the Neumann parts of the boundary. Our definition of $\mathcal{O}_{d,L}$ allows us to define an extension operator which for any $\Omega \in \mathcal{O}_{d,L}$ extends any $u \in H_{0,\Gamma^{-}}^{1}(\Omega)$ by zero below $\Gamma^{-}$ and `into $\mathcal{H}^{1}$ above $\Gamma^{+}$'. Moreover, we can uniformly bound such operators over $\mathcal{O}_{d,L}$. For a precise formulation of this see Corollary \ref{cor:Zaremba_extension_operator}. Then by similar arguments to Chenais, we prove the continuity of Zaremba eigenvalues over the collection.

\begin{thm}[Stein {\cite[\S VI]{Stein-1970}}]
    There exists a constant $C_{d,L} > 0$ depending only on $d\geq 2$ and $L>0$ such that for any $L$-Lipschitz function $f: \mathbb{R}^{d-1} \to \mathbb{R}$, there exists a Sobolev extension operator $\mathcal{E}: \mathcal{H}^{1}(\Omega_{f}) \to \mathcal{H}^{1}(\mathbb{R}^{d})$, where \begin{equation}\label{eqn:special_Lipschitz_domain}
        \Omega_{f} := \lbrace (x',y) \in \mathbb{R}^{d-1} \times \mathbb{R} : y < f(x') \rbrace,
    \end{equation} with $\lVert \mathcal{E}[u] \rVert_{\mathcal{L}^{2}} \leq C_{d,L}\lVert u \rVert_{\mathcal{L}^{2}}$ and $\lVert \mathcal{E}[u] \rVert_{\mathcal{H}^{1}} \leq C_{d,L}\lVert u \rVert_{\mathcal{H}^{1}}$ for any $u\in \mathcal{H}^{1}(\Omega_{f})$. Explicitly the extension operator is given by \begin{equation}\label{eqn:Stein_extension}
        \mathcal{E}[u](x',y) := \begin{cases} u(x',y), & (x',y) \in \Omega_{f}, \\
            \displaystyle \int_{1}^{\infty} dt \, u(x',y-td^{*}(x',y)) \Psi(t), & (x',y) \not\in \Omega_{f},
        \end{cases}
    \end{equation} where $d^{*} \geq 0$ is the regularised distance of a point to $\Omega_{f}$ and $\Psi(t)$ is a rapidly decreasing function as $t \to +\infty$.
\end{thm}

    We will refer to the above as Stein's extension operator. For a detailed account of the above we refer the reader to the monograph of Stein \cite[\S VI]{Stein-1970}. However, a detailed analysis of this Sobolev extension operator is not necessary for our means and thus we omit any further study of this. The only important point for us here is the following immediate corollary.

\begin{cor}\label{cor:Zaremba_extension_operator}
    There exists a constant $C_{d,L} > 0$ depending only on $d\geq 2$ and $L>0$ such that for any $\Omega \in \mathcal{O}_{d,L}$ there exists an extension operator $\mathcal{E}_{\Omega}:\mathcal{H}_{0,\Gamma^{-}}^{1}(\Omega) \to \mathcal{H}_{0}^{1}(\Omega_{\infty})$, where \begin{equation}
        \Omega_{\infty} = \lbrace (x',y) \in \wp(\Omega) \times \mathbb{R} : y > h_{-}(x')\rbrace, 
    \end{equation} with $\lVert \mathcal{E}_{\Omega}[u] \rVert_{\mathcal{L}^{2}} \leq C_{d,L}\lVert u \rVert_{\mathcal{L}^{2}}$ and $\lVert \mathcal{E}_{\Omega}[u] \rVert_{\mathcal{H}^{1}} \leq C_{d,L}\lVert u \rVert_{\mathcal{H}^{1}}$ for any $u\in \mathcal{H}_{0,\Gamma^{-}}^{1}(\Omega)$.
\end{cor}

\begin{proof}
    Take any $\phi \in C_{0,\Gamma^{-}}^{\infty}(\Omega) \cap C^{\infty}(\overline{\Omega})$. By a theorem of McShane in \cite{McShane-1934}, we can extend $h^{+}:\wp(\Omega) \to \mathbb{R}$ to an $L$-Lipschitz function $\widetilde{h}^{+} : \mathbb{R}^{d-1} \to \mathbb{R}$. Defining $\Omega_{\widetilde{h}^{+}}$ as in \eqref{eqn:special_Lipschitz_domain}, by extending by zero $\phi\in \mathcal{H}^{1}(\Omega_{\widetilde{h}_{+}})$, and by the definition of $\mathcal{E}$ in \eqref{eqn:Stein_extension} it is clear that one must have $\mathcal{E}[\phi] \in \mathcal{H}_{0}^{1}(\Omega_{\infty})$. Define $\mathcal{E}_{\Omega}[\phi]$ in this way. Then by the density of $C_{0,\Gamma^{-}}^{\infty}(\Omega)\cap C^{\infty}(\overline{\Omega})$ in $\mathcal{H}_{0,\Gamma^{-}}^{1}(\Omega)$, the result immediately follows.
\end{proof}

\begin{lemma}\label{lem:Zaremba_continuity}
    For each $k\in \N$, if $\Omega_{n}\in \O_{d,L}$ Hausdorff converges to $\Omega \in \mathcal{O}_{d,L}$ as $n\to +\infty$ then $\zeta_{k}(\Omega_{n}) \to \zeta_{k}(\Omega)$ as $n\to+\infty$.
\end{lemma}

\begin{proof}
    Since we know that $\Omega_{n}$ Hausdorff converges to $\Omega$, we know that there exists $\beta_{n} \to 1$ such that $\beta_{n}\Omega_{n} \subseteq \Omega$, up to a possible translation, for $n$ sufficiently large. From here onwards, we follow the ideas of the proof of Proposition IV.1 in \cite{Chenais-1975}. Fix $f\in \mathcal{L}^{2}(\Omega)$. By the Riesz-Fr\'{e}chet representation theorem there exists a unique $u_{n} \in \mathcal{H}_{0,\Gamma^{-}}^{1}(\beta_{n}\Omega_{n})$ such that \begin{equation}
        \int_{\Omega}\mathds{1}_{\beta_{n}\Omega_{n}}\nabla u_{n} \cdot \nabla \phi + \int_{\Omega}\mathds{1}_{\beta_{n}\Omega_{n}}u_{n}\phi = \int_{\Omega} \mathds{1}_{\beta_{n}\Omega_{n}} f\phi, \quad \forall \phi \in C_{0,\Gamma^{-}}^{\infty}(\beta_{n}\Omega_{n})
    \end{equation} with $\lVert u_{n} \rVert_{\mathcal{H}^{1}(\beta_{n}\Omega_{n})} = \lVert f\rVert_{\mathcal{L}^{2}(\beta_{n}\Omega_{n})} \leq \lVert f\rVert_{\mathcal{L}^{2}(\Omega)}.$ Then we see that we can extend each $u_{n}\in \mathcal{H}_{0,\Gamma^{-}}^{1}(\beta_{n}\Omega_{n})$ via $\mathcal{E}_{\Omega_{n}}$, as defined in Corollary \ref{cor:Zaremba_extension_operator}, to a function $\bar{u}_{n} \in \mathcal{H}_{0,\Gamma^{-}}^{1}(\Omega)$ with $\lVert \bar{u}_{n} \rVert_{\mathcal{H}^{1}(\Omega)} \leq C_{d,L}\lVert f\rVert_{\mathcal{L}^{2}(\Omega)}$. By the Banach-Alaoglu theorem, up to a subsequence, $\bar{u}_{n} \rightharpoonup u$ in $\mathcal{H}_{0,\Gamma}^{1}(\Omega)$ as $n\to +\infty$. We now show that $u$ must be the unique solution to 
    \begin{equation}\label{eqn:weak_poisson}
        \int_{\Omega}\nabla u \cdot \nabla \phi + \int_{\Omega}u\phi = \int_{\Omega} f\phi, \quad \forall \phi \in C_{0,\Gamma^{-}}^{\infty}(\Omega).
    \end{equation}

    Fix $\phi \in C_{0,\Gamma^{-}}^{\infty}(\Omega)$. Then by Lemma \ref{lem:Lip_boundary}, we see that the support of $\phi$ is at a positive distance from $\Gamma^{-}(\Omega_{n})$ for $n$ sufficiently large, and so $\left.\phi\right|_{\beta_{n}\Omega_{n}} \in \mathcal{H}_{0,\Gamma^{-}}^{1}(\beta_{n}\Omega_{n})$ for $n$ sufficiently large. Thus, for $n$ sufficiently large \begin{equation}
        \int_{\Omega}\mathds{1}_{\beta_{n}\Omega_{n}}\nabla \bar{u}_{n} \cdot \nabla \phi + \int_{\Omega}\mathds{1}_{\beta_{n}\Omega_{n}}\bar{u}_{n}\phi = \int_{\Omega} \mathds{1}_{\beta_{n}\Omega_{n}} f\phi.
    \end{equation} 

    Following the arguments in \cite[Prop IV.1]{Chenais-1975}, it is clear that if we take the limit $n\to+\infty$, \begin{equation}
        \int_{\Omega}\nabla u \cdot \nabla \phi + \int_{\Omega}u\phi = \int_{\Omega} f\phi.
    \end{equation} Since $\phi\in C_{0,\Gamma^{-}}^{\infty}(\Omega)$ was arbitrary, $u$ is indeed the solution to \eqref{eqn:weak_poisson} as desired. Moreover, $\bar{u}_{n}\to u$ in $\mathcal{L}^{2}(\Omega)$ by the Rellich–Kondrachov compactness theorem since $\bar{u}_{n} \rightharpoonup u$ in $\mathcal{H}_{0,\Gamma^{-}}^{1}(\Omega)$. Now following the proof of Theorem 2.3.2. in \cite{Henrot-2006}, we see that $\zeta_{k}(\beta_{n}\Omega_{n}) \to \zeta_{k}(\Omega)$. Then noting that $\zeta_{k}(\beta_{n}\Omega_{n}) = (\beta_{n})^{-2}\zeta_{k}(\Omega_{n})$ we obtain the result.
\end{proof}

\subsection{Proof of Theorems \ref{thm:Zaremba_surface_measure} and \ref{thm:Zaremba_diamater}}
With the continuity of Zaremba eigenvalues over $\mathcal{O}_{d,L}$ now in hand, we can now prove the existence of minimisers using Stein's extension operator.

\begin{lemma}\label{lem:existence}
    For each $k\geq 1$ there exists a minimiser $\Omega_{k}^{*}$ to \eqref{eq:Zaremba_perim_problem}.
\end{lemma}

\begin{proof}
    Let $\delta = \lVert h^{+}-h^{-} \rVert_{\infty}$. Then one sees that, up to a possible translation, $\Omega \subset \wp(\Omega) \times (0,\delta)$. We can extend the first Zaremba eigenfunction to the Sobolev space $\mathcal{H}_{0,\wp(\Omega)\times \lbrace 0\rbrace}^{1}(\wp(\Omega) \times (0,\delta))$. Hence, we see that \begin{equation}
        \widetilde{\zeta}_{1}(\wp(\Omega) \times (0,\delta)) \leq C_{d,L} \zeta_{1}(\Omega)
    \end{equation} where $\widetilde{\zeta}_{1}(\wp(\Omega) \times (0,\delta))$ is the $k$-th eigenvalue of the Zaremba Laplacian on $\wp(\Omega) \times (0,\delta)$ with Dirichlet boundary conditions on $\wp(\Omega)\times \lbrace 0\rbrace$. By separation of variables one can deduce that \begin{equation}
        \widetilde{\zeta}_{1}(\wp(\Omega) \times (0,\delta)) = \mu_{1}(\wp(\Omega)) + \frac{\pi^{2}}{4\delta^{2}} = \frac{\pi^{2}}{4\delta^{2}}.
    \end{equation} And so we see that \begin{equation}
        \zeta_{k}(\Omega) \geq \zeta_{1}(\Omega) \geq \frac{\pi^{2}}{4C_{d,L} \delta^{2}} \uparrow +\infty
    \end{equation} as $\delta \downarrow 0$. Hence, we must have that $\delta$ is uniformly bounded and so the inradii of the sets must be uniformly bounded. Let $\Omega_{n}$ be a minimising sequence for the infimum, then since the inradius is uniformly bounded from below we have that, up to a sequence of translations, a Hausdorff convergent subsequence $\Omega_{n_{j}}$ converging to some domain $\Omega \in \mathcal{O}_{d,L}$ as $j\to +\infty$ by Blaschke's selection theorem and Lemma \ref{lem:Odl_closed}. Since the Zaremba eigenvalues are continuous in this topology, see Lemma \ref{lem:Zaremba_continuity}, $\zeta_{k}(\Omega_{n_{j}}) \to \zeta_{k}(\Omega)$ as $j\to +\infty$ and we are done.
    \end{proof}

We now give a lower bound for Zaremba eigenvalues in the spirit of the classical Li-Yau bound, see \cite[Cor. 1]{Li-Yau-1983}, for Dirichlet eigenvalues.

\begin{lemma}\label{thm:Li-Yau}
There exists a constant $C_{d,L} > 0$, depending only on $d\geq 2$ and $L>0$, such that for any $0<\alpha < 1/d$ \begin{equation}
    \zeta_{k}(\Omega) \geq \frac{C_{d,L}k^{2/d}}{(|\Omega|+k^{-\alpha})^{2/d}} - |\wp(\Omega)|^{2}k^{2\alpha} - 1 - (d-1)L^{2}
\end{equation} for all $\Omega \in \mathcal{O}_{d,L}$.
\end{lemma}

\begin{proof}
Let $\epsilon > 0$. Fix $\Omega \in \mathcal{O}_{d,L}$ and define the set \begin{equation}
    \Omega^{\epsilon} = \lbrace (x',y) \in \wp(\Omega) \times \mathbb{R} : h^{-}(x') < y < h^{+}(x') + \epsilon\rbrace.
\end{equation} Further for $\epsilon >0$, define the function $\chi_{\epsilon}: \wp(\Omega) \times \mathbb{R} \to [0,1]$ by \begin{equation}
    \chi_{\epsilon}(x',y) := \begin{cases}
        1, & y \leq h^{+}(x'), \\
        1-\frac{(y-h^{+}(x'))}{\epsilon}, & h^{+}(x') < y < h^{+}(x') +\epsilon, \\
        0 & y \geq h^{+}(x') + \epsilon
    \end{cases}
\end{equation} Let $\mathcal{E}$ be the Sobolev extension operator given in Corollary \ref{cor:Zaremba_extension_operator}. For any $u\in \mathcal{H}_{0,\Gamma^{-}}(\Omega)$, we have that $\chi_{\epsilon}\mathcal{E}[u] \in \mathcal{H}_{0}^{1}(\Omega^{\epsilon})$. Moreover, let $S_{k} = \lbrace u_{1},\ldots, u_{k}\rbrace$ denote the span of the first $k$ orthonormal eigenfunctions of $-\Delta_{\Omega}^{Z}$. Then the collection $\lbrace \chi_{\epsilon}\mathcal{E}[u_{1}],\ldots, \chi_{\epsilon}\mathcal{E}[u_{k}]\rbrace \subset \mathcal{H}_{0}^{1}(\Omega^{\epsilon})$ is linearly independent and so we pass this as a trial space into the variational formulation for the $k$-th Dirichlet eigenvalue of $\Omega^{\epsilon}$. By repeated use of the uniform bound given in Corollary \ref{cor:Zaremba_extension_operator}, we have \begin{equation}
\begin{split}
    \lambda_{k}(\Omega^{\epsilon}) & \leq \max_{\substack{u\in S_{k} \\ \lVert u\rVert_{\mathcal{L}^{2}(\Omega)}=1}} \int_{\Omega^{\epsilon}} \left|\nabla \left(\chi_{\epsilon}\mathcal{E}[u]\right)\right|^{2} \\
    & \leq \max_{\substack{u\in S_{k} \\ \lVert u\rVert_{\mathcal{L}^{2}(\Omega)}=1}} \left\lbrace \int_{\Omega} \left|\nabla u\right|^{2} + \int_{\Omega^{\epsilon}\backslash \Omega} \left|\chi_{\epsilon}\nabla \mathcal{E}[u]+\mathcal{E}[u]\nabla \chi_{\epsilon} \right|^{2}\right\rbrace \\
    & \leq \max_{\substack{u\in S_{k} \\ \lVert u\rVert_{\mathcal{L}^{2}(\Omega)}=1}} \left\lbrace \int_{\Omega} \left|\nabla u\right|^{2} + 2\int_{\Omega^{\epsilon}\backslash \Omega} \left|\chi_{\epsilon}\nabla \mathcal{E}[u]\right|^{2} + 2\int_{\Omega^{\epsilon}\backslash \Omega} \left|\mathcal{E}[u]\nabla \chi_{\epsilon} \right|^{2}\right\rbrace \\
    & \leq \max_{\substack{u\in S_{k} \\ \lVert u\rVert_{\mathcal{L}^{2}(\Omega)}=1}} \left\lbrace \int_{\Omega} \left|\nabla u\right|^{2} + 2\int_{\Omega^{\epsilon}\backslash \Omega} \left|\nabla \mathcal{E}[u]\right|^{2} + 2((d-1)L^{2}+\epsilon^{-2})\int_{\Omega^{\epsilon}\backslash \Omega} \left|\mathcal{E}[u]\right|^{2}\right\rbrace \\
    & \leq \max_{\substack{u\in S_{k} \\ \lVert u\rVert_{\mathcal{L}^{2}(\Omega)}=1}} \left\lbrace (1+2C_{d,L})\int_{\Omega} \left|\nabla u\right|^{2} + 2C_{d,L}+ 2C_{d,L}((d-1)L^{2}+\epsilon^{-2}) \right\rbrace \\
    & \leq C_{d,L}' \left( \max_{\substack{u\in S_{k} \\ \lVert u\rVert_{\mathcal{L}^{2}(\Omega)}=1}} \left\lbrace \int_{\Omega} \left|\nabla u\right|^{2} \right\rbrace + (1+(d-1)L^{2}+\epsilon^{-2})\right) \\ 
    & = C_{d,L}' (\zeta_{k}(\Omega) +1+ (d-1)L^{2}+\epsilon^{-2}).
\end{split}
\end{equation} By the classical Dirichlet eigenvalue lower bound of Li and Yau \cite[Cor. 1]{Li-Yau-1983}, we see that \begin{equation}
    \lambda_{k}(\Omega^{\epsilon}) \geq \frac{dW_{d}k^{2/d}}{(d+2)|\Omega^{\epsilon}|^{2/d}}.
\end{equation} Now, observing that $|\Omega^{\epsilon}|=|\Omega|+\epsilon|\wp(\Omega)|$, we obtain that \begin{equation}
    \zeta_{k}(\Omega) \geq \frac{C_{d,L}''k^{2/d}}{(|\Omega|+\epsilon |\wp(\Omega)|)^{2/d}} - 1 -(d-1) L^{2} - \epsilon^{-2}.
\end{equation}
    Taking $\epsilon = |\wp(\Omega)|^{-1}k^{-\alpha}$, then the result immediately follows.
\end{proof}

Before proving Theorem \ref{thm:Zaremba_surface_measure}, we now briefly look at the isoperimetric problem \begin{equation}\label{eq:Odl_isoperim}
    \sup \left\lbrace |\Omega| : \Omega \in \O_{d,L}, \, |\partial \Omega| = 1 \right\rbrace
\end{equation} for domains in $\mathcal{O}_{d,L}$. It is clear that there exists a solution to the isoperimetric problem over $\O_{d,L}$, however we cannot say too much immediately as balls do not lie in $\mathcal{O}_{d,L}$. We now give some remarks on properties of solutions to \eqref{eq:Odl_isoperim}.

By the results of Fuglede in \cite{Fuglede-1989}, for $L >0$ large one can note that any solution to the isoperimetric problem must be (quantifiably) close to the ball of the same perimeter. Moreover, for any $\Omega\in \O_{d,L}$, its Steiner symmetrisation $\Omega^{\#}$ about the hyperplane $\lbrace x_{d} = 0\rbrace $ defined by \begin{equation}
    \Omega^{\#} := \left\lbrace (x',y) \in \wp(\Omega) \times \R : -h(x') < y < h(x')\right\rbrace,
\end{equation} where $h(x'):= (h^{+}(x')-h^{-}(x'))/2$, also lies in $\O_{d,L}$. Thus, we have that $|\Omega^{\#}|=|\Omega|$ and $|\partial \Omega^{\#}| \leq |\partial \Omega|$, with equality if and only if $\Omega^{\#}$ and $\Omega$ are isometric. Hence, any solution to the isoperimetric problem over $\O_{d,L}$ is necessarily symmetric about the hyperplane $\lbrace x_{d} = 0\rbrace $.

As far as the author is aware, it is not known whether the solution to the isoperimetric problem over $\O_{d,L}$ is unique. In dimension two, it appears to be unique and the author has numerically computed solutions to the isoperimetric problem for $\mathcal{O}_{2,L}$, see Figure \ref{fig:isoperimetric_solutions}.

We now show for any $\Omega \in \O_{d,L}$ that imposing the condition $|\partial \Omega|=1$ imposes a constraint on $|\wp(\Omega)|$, which is the final ingredient needed to prove Theorem \ref{thm:Zaremba_surface_measure}.

\begin{figure}
    \centering
    \includegraphics[width=7cm]{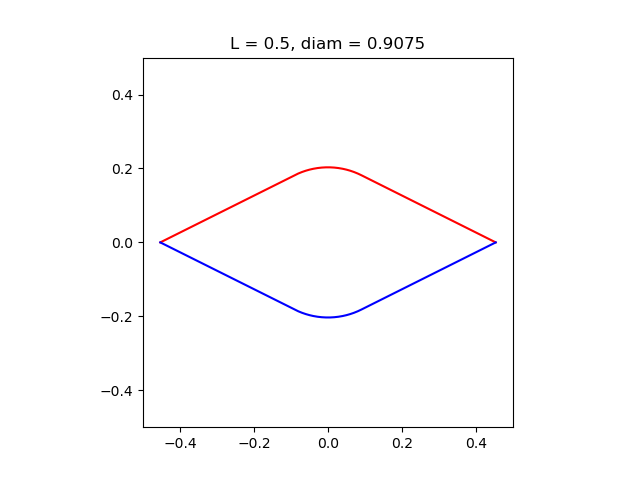}
    \includegraphics[width=7cm]{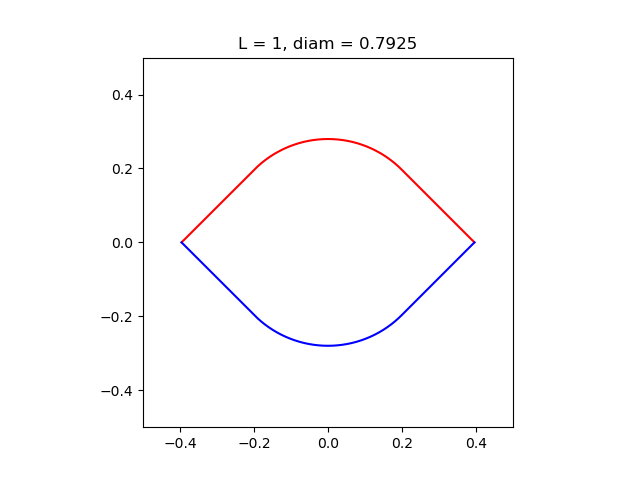}
    \includegraphics[width=7cm]{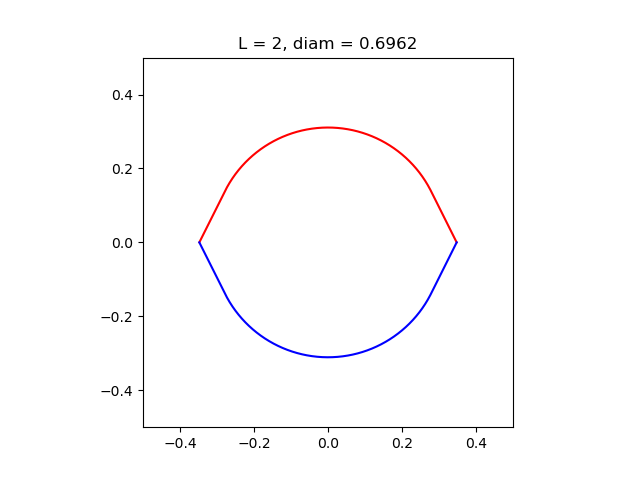}
    \includegraphics[width=7cm]{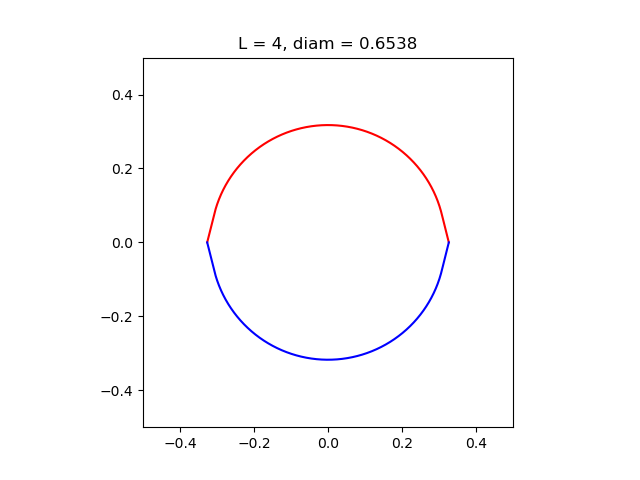}
    \caption{Numerically computed optimal solutions to the isoperimetric problem with unit perimeter over $\O_{2,L}$ with $L=0.5$ (top left), $L=1$ (top right), $L=2$ (bottom left) and $L=4$ (bottom right).}
    \label{fig:isoperimetric_solutions}
\end{figure}

\begin{lemma}\label{lem:projection_perim_ineq}
    Fix $\Omega \in \mathcal{O}_{d,L}$ and suppose that $|\partial\Omega| = 1$, then $|\wp(\Omega)|\leq \frac{1}{2}$.
\end{lemma}

\begin{proof}
    For $\Omega \in \mathcal{O}_{d,L}$ observe that \begin{equation}
        |\partial \Omega| = \int_{\wp(\Omega)} \sqrt{1+|\nabla h^{+}|^{2}} + \sqrt{1+|\nabla h^{-}|^{2}} \geq 2|\wp(\Omega)|
    \end{equation} and the result immediately follows.
\end{proof}

\begin{proof}[Proof of Theorem \ref{thm:Zaremba_surface_measure}]
With our previous results in hand, we now follow the outline of the proof of Theorem 1.1 in \cite{Bucur-Freitas-2013} to prove Theorem \ref{thm:Zaremba_surface_measure}. We already know the existence of minimisers to \eqref{eq:Zaremba_perim_problem} from Lemma \ref{lem:existence}. Let $\Omega_{k}^{*}$ be any sequence of minimisers to \eqref{eq:Zaremba_perim_problem} and let $\Omega \in \O_{d,L}$ with $|\partial\Omega| =1$ be fixed. Using Lemma \ref{lem:projection_perim_ineq} and Theorem \ref{thm:Li-Yau}, taking $\alpha = 1/2d$, we see that \begin{equation}
    \zeta_{k}(\Omega_{k}^{*}) \geq \frac{C_{d,L}k^{2/d}}{(|\Omega_{k}^{*}|+k^{-1/2d})^{2/d}}-\frac{1}{4}k^{1/d}-1-(d-1)L^{2}.
\end{equation} Then observe that \begin{equation}
    \frac{C_{d,L}k^{2/d}}{(|\Omega_{k}^{*}|+k^{-1/2d})^{2/d}}-\frac{1}{4}k^{1/d}-1-(d-1)L^{2} \leq \zeta_{k}(\Omega_{k}^{*}) \leq \zeta_{k}(\Omega) = \frac{W_{d}k^{2/d}}{|\Omega|^{2/d}} + o(k^{2/d})
\end{equation} and dividing through by $k^{2/d}$ and taking the limsup \begin{equation}
    \liminf_{k\to +\infty}|\Omega_{k}^{*}|^{2/d} \geq C_{d,L}(W_{d})^{-1}|\Omega|^{2/d} > 0.
\end{equation} Hence, the sequence of minimisers is non-degenerate. Now the only moot point to cover is that any accumulation point of the sequence $\Omega_{k}^{*}$, possibly up to translations of elements of the sequence, is indeed a solution to the isoperimetric problem over $\O_{d,L}$. Knowing the non-degeneracy, by Blaschke's selection theorem, Lemma \ref{lem:Odl_closed} and the inequalities on p. 146 of \cite{van-den-Berg-2015}, up to a sequence of translations, the $\Omega_{k}^{*}$ lie inside a sequentially compact subcollection of $\mathcal{O}_{d,L}$. Hence, there is a convergent subsequence $\Omega_{k_{j}}^{*}$, up to translating elements of the sequence, converging to some $\Omega_{\infty} \in \O_{d,L}$ as $j\to +\infty$. By Theorem \ref{thm:generalised_Weyl_law} we see that \begin{equation}
    \lim_{j\to +\infty} \frac{\zeta_{k_{j}}(\Omega_{k_{j}})}{(k_{j})^{2/d}} = \frac{W_{d}}{|\Omega_{\infty}|^{2/d}},
\end{equation} using Dirichlet-Neumann bracketing i.e. $\mu_{k}(\Omega) \leq \zeta_{k}(\Omega) \leq \lambda_{k}(\Omega)$ for $\Omega \in \mathcal{O}_{d,L}$. Now if $\Omega_{\infty}$ is not a solution to the isoperimetric problem then we see that this would violate the optimality of the sequence $\Omega_{k}^{*}$. Moreover, $\Omega_{\infty}$ is necessarily symmetric about the hyperplane $\lbrace x_{d}=0\rbrace$, up to a translation, by our previous discussion.
\end{proof}

The proof of Theorem \ref{thm:Zaremba_diamater} follows entirely analogously to the proof of Theorem \ref{thm:Zaremba_surface_measure} by noting that the condition $\diam(\Omega) =1$ implies that $|\wp(\Omega)| \leq 2^{-(d-1)} \omega_{d-1}$ via the $(d-1)$-dimensional isodiametric inequality.

\appendix

\section{An open question}\label{sec:appendix}

In this appendix, we discuss a natural question which arises from the techniques used in this paper and a connection with Pólya's conjecture.

Pólya's conjecture, first given in \cite{Pólya-1954}\footnote{As remark in a footnote on p. 132 in \cite{Filonov-Levitin-Polterovich-Sher-2023}, Pólya originally only made the conjecture for planar domains in a slightly different form.}, asserts that for a suitably regular bounded domain $\Omega\subset \mathbb{R}^{d}$ one has \begin{equation}\label{eq:Polya}
    \mathcal{N}_{\Omega}^{D}(\alpha) \leq \frac{\omega_{d}|\Omega|}{(2\pi)^{d}} \alpha^{d/2} \leq \mathcal{N}_{\Omega}^{N}(\alpha).
\end{equation} We call the inequalities on left and right hand side of \eqref{eq:Polya} `Pólya's conjecture for the Dirichlet Laplacian' and `Pólya's conjecture for the Neumann Laplacian' respectively. The inequalities in \eqref{eq:Polya} are known to hold for tiling domains \cite{Polya-1961,Kellner-1966}. It was recently shown in \cite{Filonov-Levitin-Polterovich-Sher-2023} that Pólya's conjecture is true for the Dirichlet Laplacian on balls in any dimension and for the Neumann Laplacian for balls in dimension two. We refer the reader to \cite{Filonov-Levitin-Polterovich-Sher-2023} for a fuller discussion of the conjecture and its history.

The relationship between Pólya's conjecture and asymptotic shape optimisation problems has been explored in \cite{Colbois-ElSoufi-2014} and \cite{Freitas-Lagace-Payette-2021}. Our method of proving Propositions \ref{prop:Neumann_counting function_upper_bound} and \ref{prop:Dir_counting_lower} does lead to a natural question in this direction. The question in itself can be thought of as a possible conjecture concerning a refinement of Dirichlet-Neumann bracketing.

\begin{question}\label{ques:DN-Polya}
    Let $\delta > 0$, and set $Q_{\delta}=(0,\delta)\times (0,\delta)$ and $\Gamma \subset \mathbb{R}^{2}$ be a simple closed smooth curve which partitions $Q_{\delta}$ into two connected non-empty open subsets, which we denote $\Omega_{1}$ and $\Omega_{2}$. Prescribing Neumann boundary conditions on $\partial Q$ and Dirichlet boundary conditions on $\Gamma$, consider the Zaremba Laplacians $-\Delta_{\Omega_{1}}$ and $-\Delta_{\Omega_{2}}$ acting on $\mathcal{L}^{2}(\Omega_{1})$ and $\mathcal{L}^{2}(\Omega_{2})$ with the boundary conditions described. See Figure \ref{fig:DN-Polya} for an illustration of this. Denote their respective eigenvalue counting functions by $\mathcal{N}_{\Omega_{1}}$ and $\mathcal{N}_{\Omega_{2}}$. Then is it true that \begin{equation}\label{eq:DN-Polya}
        \mathcal{N}_{\Omega_{j}}(\alpha) \leq \frac{|\Omega_{j}|}{|Q_{\delta}|} \mathcal{N}_{Q_{\delta}}^{N}(\alpha)
    \end{equation} for all $\alpha > 0$ and $j=1,2$?
\end{question}

\begin{figure}
    \centering
    \begin{tikzpicture}
        \draw[red,thick] (0,0) -- (3,0) -- (3,3) -- (0,3) -- (0,0);
        \draw[blue, thick] (0,1) to[out=-20,in=100] (3,2);
        \draw[,xshift=4cm,red,thick] (0,0) -- (3,0) -- (3,3) -- (0,3) -- (0,0);
        \draw[xshift=4cm, blue, thick] (0.5,1) to[out=0,in=100] (2.5,2);
        \draw[xshift=4cm, blue, thick] (2.5,2) to[out=280,in=0] (1,0.5);
        \draw[xshift=4cm, blue, thick] (1,0.5) to[out=180,in=180] (0.5,1);
        \node at (1.5,0.75) {$\Omega_{1}$};
        \node at (2,1.5) {\color{blue} $\Gamma$};
        \node at (1.5,2.25) {$\Omega_{2}$};
        \node at (1.5,-0.5) {(i) $\Gamma \cap Q_{\delta} \neq \Gamma$};
        \node[xshift=4cm] at (1.5,-0.5) {(ii) $\Gamma \subset Q_{\delta}$};
        \node[xshift=4cm] at (2,2.4) {\color{blue} $\Gamma$};
        \node[xshift=4cm] at (1.7,1.25) {$\Omega_{1}$};
        \node[xshift=4cm] at (1,2) {$\Omega_{2}$};
    \end{tikzpicture}
    \caption{Illustation of the two cases arising from Question \ref{ques:DN-Polya} with $\partial Q_{\delta}$ shown in red and $\Gamma$ shown in blue.}
    \label{fig:DN-Polya}
\end{figure}
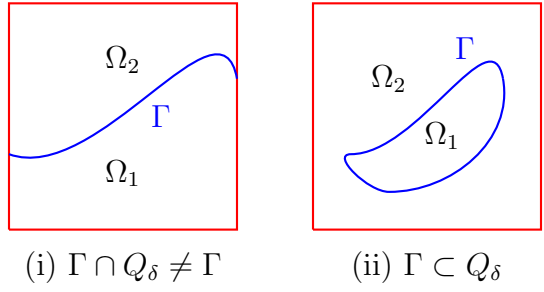

As a remark, in Question \ref{ques:DN-Polya}, if $\Gamma \subset Q_{\delta}$, see for example Figure \ref{fig:DN-Polya}(ii), then the inequality in \eqref{eq:DN-Polya} is weaker than that of Pólya's conjecture for the Dirichlet Laplacian of the domain enclosed by $\Gamma$, as we know Pólya's conjecture holds for squares. It is also worth noting that one only needs to consider the case where $\delta = 1$ by the scaling properties of eigenvalues under homothety.

We now outline how one can prove the two-dimensional Pólya conjecture for the Dirichlet Laplacian for connected smooth, not-necessarily convex, domains, assuming that the inequality \eqref{eq:DN-Polya} holds. As in the proof of Proposition \ref{prop:Neumann_counting function_upper_bound}, fix $n\in \mathbb{N}$ and let $\mu_{n+1}^{*}$ be the $(n+1)$-th Neumann eigenvalue of the unit square. Let $\Omega \subset \mathbb{R}^{2}$ be a smooth bounded domain and fix some $\delta > 0$. Setting $Q_{m,\delta} := \delta(m+(0,1)^{2})$ for $m\in \mathbb{Z}^{2}$, define the collections \begin{equation}
    \mathcal{I}_{\delta} := \lbrace m\in \mathbb{Z}^{2} : Q_{m,\delta} \subset \Omega \rbrace
\end{equation} and \begin{equation}
    \mathcal{J}_{\delta} := \lbrace m\in \mathbb{Z}^{2} : Q_{m,\delta} \cap \Omega \neq \emptyset, m\not\in \mathcal{I}_{\delta}  \rbrace.
\end{equation} By usual bracketing methods one has that \begin{equation}
\begin{split}
    \mathcal{N}_{\Omega}^{D}(\delta^{-2}\mu_{n+1}^{*}) & \leq \sum_{m\in \mathcal{I}_{\delta}} \mathcal{N}_{Q_{m,\delta}}^{N}(\delta^{-2}\mu_{n+1}^{*}) + \sum_{m \in \mathcal{J}_{\delta}} \mathcal{N}_{Q_{m,\delta}\cap \Omega}^{Z}(\delta^{-2}\mu_{n+1}^{*}) \\
    & {\color{red}\overset{(*)}{\leq}} \sum_{m\in \mathcal{I}_{\delta}} \mathcal{N}_{Q_{m,\delta}}^{N}(\delta^{-2}\mu_{n+1}^{*}) + \sum_{m \in \mathcal{J}_{\delta}} \frac{|Q_{m,\delta} \cap \Omega|}{|Q_{m,\delta}|}\mathcal{N}_{Q_{m,\delta}}^{N}(\delta^{-2}\mu_{n+1}^{*}) \\
    & \leq \sum_{m\in \mathcal{I}_{\delta}} n + \sum_{m \in \mathcal{J}_{\delta}} \frac{n|Q_{m,\delta} \cap \Omega|}{|Q_{m,\delta}|} \\
    & \leq n\delta^{-2}\left|\bigcup_{m\in \mathcal{I}_{\delta}} Q_{m,\delta}\right| + n\delta^{-2} \left|\Omega \cap \bigcup_{m\in \mathcal{J}_{\delta}} Q_{m,\delta}\right| \\
    & = n\delta^{-2}|\Omega|,
\end{split}
\end{equation} where we assume a positive answer to Question \ref{ques:DN-Polya} in the inequality highlighted in red in the above. Now taking $\delta = \sqrt{\alpha^{-1}\mu_{n+1}^{*}}$ for $\alpha > 0$, we see that  \begin{equation}
    \mathcal{N}_{\Omega}^{D}(\alpha) \leq \frac{n|\Omega|}{\mu_{n+1}^{*}} \alpha.
\end{equation} Taking the limit as $n\to +\infty$, by Weyl's law we have that \begin{equation}
    \mathcal{N}_{\Omega}^{D}(\alpha) \leq \frac{|\Omega|}{4\pi} \alpha
\end{equation} which is the conjectured bound of Pólya for the Dirichlet Laplacian in dimension two.

One can also ask similar questions to Question \ref{ques:DN-Polya} where one instead puts Neumann boundary conditions on $\Gamma$ and Dirichlet boundary conditions on $\partial Q$, and ask if \begin{equation}
    \mathcal{N}_{\Omega_{j}}(\alpha) \geq \frac{|\Omega_{j}|}{|Q_{\delta}|} \mathcal{N}_{Q_{\delta}}^{D}(\alpha)
\end{equation} for any $\alpha > 0$ and $j = 1,2$. This links to Pólya's conjecture for the Neumann Laplacian in a similar way to the above following the proof of Proposition \ref{prop:Dir_counting_lower}. Of course, one can ask analogous questions to Question \ref{ques:DN-Polya} in higher dimensions.

\vspace{1em}
\textbf{Acknowledgements:} The author is immensely grateful to his PhD supervisor Katie Gittins for numerous suggestions and stimulating discussions on the work in this paper, in particular for referring him to the Neumann eigenvalue counting function upper bound in \cite{Gittins-Lena-2020}. The author also gratefully acknowledges the support of his Engineering and Physical Sciences Research Council Doctoral Training Grant [Grant Number: EP/T518001/1] whilst carrying out this work. \\

\textbf{Data Access Statement}:
This work did not involve any underlying data.

\bibliographystyle{alphaurl}
\bibliography{biblio}
\end{document}